\documentclass{mystyle}
\usepackage[utf8]{inputenc}
\usepackage{amsmath}
\usepackage{calc}
\usepackage[dvipsnames]{xcolor}
\usepackage{accents}
\usepackage[all]{xy}   
\usepackage{tikz-cd}
\usepackage{cancel}
\tikzcdset{arrow style=tikz, diagrams={>=stealth}}                        %
  
\CompileMatrices                            

\UseTips                                    

\input xypic
\usepackage[bookmarks=true]{hyperref}       

\usepackage{amssymb,latexsym,amsmath,amscd}
\usepackage{xspace}
\usepackage{color}
\usepackage{graphicx}
\usepackage{dsfont}

\reversemarginpar

\DeclareMathOperator{\Res}{Res}

\theoremstyle{plain}
\newtheorem{theorem}{Theorem}[section]
\newtheorem*{theorem*}{Theorem}
\newtheorem{proposition}[theorem]{Proposition}
\newtheorem{corollary}[theorem]{Corollary}
\newtheorem{lemma}[theorem]{Lemma}

\theoremstyle{definition}

\newtheorem{remark}[theorem]{Remark}

\newtheorem*{mainresults}{Main Results}

\newcommand{\enm}[1]{\ensuremath{#1}}          %
\newcommand{\op}[1]{\operatorname{#1}}
\newcommand{\cal}[1]{\mathcal{#1}}

\renewcommand{\bar}[1]{\overline{#1}}

\newcommand{\EE}{\enm{\mathbb{E}}}

\newcommand{\NN}{\enm{\mathbb{N}}}

\newcommand{\ZZ}{\enm{\mathbb{Z}}}
\newcommand{\FF}{\enm{\mathbb{F}}}
\renewcommand{\AA}{\enm{\mathbb{A}}}
\newcommand{\PP}{\enm{\mathbb{P}}}

\newcommand{\KK}{\enm{\mathbb{K}}}

\newcommand{\Dd}{\enm{\cal{D}}}
\newcommand{\Ee}{\enm{\cal{E}}}

\newcommand{\Gg}{\enm{\cal{G}}}
\newcommand{\Hh}{\enm{\cal{H}}}
\newcommand{\Ii}{\enm{\cal{I}}}

\newcommand{\Ll}{\enm{\cal{L}}}
\newcommand{\Mm}{\enm{\cal{M}}}
\newcommand{\Nn}{\enm{\cal{N}}}
\newcommand{\Oo}{\enm{\cal{O}}}

\newcommand{\Rr}{\enm{\cal{R}}}

\newcommand{\Tt}{\enm{\cal{T}}}
\newcommand{\Uu}{\enm{\cal{U}}}

\newcommand{\Xx}{\enm{\cal{X}}}

\newcommand{\Zz}{\enm{\cal{Z}}}

\renewcommand{\phi}{\varphi}
\renewcommand{\theta}{\vartheta}
\renewcommand{\epsilon}{\varepsilon}

\newcommand{\Aut}{\op{Aut}}

\newcommand{\codim}{\op{codim}}

\renewcommand{\to}[1][]{\xrightarrow{\ #1\ }}

\renewcommand{\sl}{\textsl{sl}}

\newcommand{\old}[1]{}
\newcommand{\vni}{\vskip 1.5pt \noindent}

\newcommand{\w}{\widetilde}
\newcommand{\h}{\ensuremath{\mathcal}}
\newcommand{\HL}{\ensuremath{\mathcal{H}^\mathcal{L}_}}
\newcommand{\HO}{\ensuremath{\mathcal{H}^\mathcal{}_}}
\newcommand{\ce}{C_{\mathcal{E}}}
\newcommand{\wce}{\widetilde{C_{\mathcal{E}}}}
\newcommand{\hce}{\widehat{C_{\mathcal{E}}}}

\begin{document}
\title[On the Hilbert scheme of smooth curves in $\mathbb{P}^5$]
{On the Hilbert scheme of smooth curves of degree $d=15$ in $\mathbb{P}^5$}

\thanks{The first named author is a member of GNSAGA of INdAM (Italy). The second named author was supported in part by the National Research Foundation of South Korea 
 (2022R1I1A1A01055306).} 

\author[E. Ballico]{Edoardo Ballico}
\address{Dipartimento di Matematica, Universita` degli Studi di Trento\\
Via Sommarive 14, 38123 Povo, Italy}
\email{ballico@science.unitn.it
}

\author[C. Keem]{Changho Keem}
\address{
Department of Mathematics,
Seoul National University\\
Seoul 151-742,  
South Korea}
\email{ckeem1@gmail.com \textrm{and} ckeem@snu.ac.kr}

\thanks{}



\date{\today}
\maketitle

\begin{abstract}
We denote by $\mathcal{H}_{d,g,r}$ the Hilbert scheme of smooth curves, which is the union of components whose general point corresponds to a smooth, irreducible and non-degenerate curve of degree $d$ and genus $g$ in $\mathbb{P}^r.$
In this article, we study $\mathcal{H}_{15,g,5}$ for every possible genus $g$ and determine  when it is irreducible. We also study the moduli map $\mathcal{H}_{15,g,5}\rightarrow\mathcal{M}_g$
and several key properties such as  gonality of a general element as well as characterizing smooth elements of each component.
\end{abstract}
\vspace{-8pt}
\section{Introduction}
In this paper we study the Hilbert scheme of smooth and non-degenerate curves $X\subset \PP^5$ of degree $15$.

Let  $\h{H}_{d,g,r}$ denote the Hilbert scheme of smooth curves of degree $d$ and genus $g$ in $\PP^r$.
We determine the number of irreducible components, their dimensions and study the properties of $\h{H}_{15,g,5}$ such as the gonality of a general element in each component $\h{H}\subset \h{H}_{15,g,5}$. We also study the natural functorial map $\mu : \h{H} \rightarrow\h{M}_g$.

Severi \cite{Sev} asserted that 
the Hilbert scheme  $\h{H}_{d,g,r}$ is irreducible for 
triples $(d,g,r)$ in the range 

(i) $d\ge g+r~~~~$  

\noindent

\noindent
or in the much wider Brill-Noether range

(ii) $\rho(d,g,r):=g-(r+1)(g-d+r)\ge 0$.  

In general, determining the irreducibility of a Hilbert scheme is a non-trivial task.
The assertion of Severi turns out to be true for $r=3, 4$ under the condition (i); cf. \cite{E1, E2}. 
However our overall knowledge on $\HO{d,g,5}$ is not as  extensive as those on $\HO{d,g,3}$ or $\HO{d,g,4}$.

In this paper, we would like to concentrate on $\HO{d,g,5}$ when the degree $d$ of the curve in $\PP^5$ is relatively low. Specifically,  we focus our attention on curves in $\PP^5$ of degree $d=15$ and determine when $\HO{15,g,5}$ is irreducible for every possible genus $g\le \pi(15,5)=18$. Our results include:
\vspace{-1pt}
\noindent
\begin{mainresults}\label{mainresults}
\begin{enumerate}
\item[(i)] $\HO{15,g,5}$ is irreducible unless $g=13, 14, 16$; Propositions \ref{g=13}, \ref{main12}, \ref{main11}, \ref{10less}, \ref{maing=18}, Theorems \ref{main15}, \ref{main16} and  \cite[Theorem 1.1]{bumi}.
\item[(ii)] $\HO{15,g,5}=\emptyset$ for $g=17$; Proposition \ref{main17}.
\item[(iii)] An estimate of the dimension of the image of the forgetful map $\h{H}\rightarrow\h{M}_g$ 
and the gonality of curves in each component $\h{H}\subset\HO{15,g,5}$. 
\item[(iv)] A general $X \in \HO{15,15,5}$ has exactly one $g^2_8$, one base-point-free $g^1_5$ and 3 base-point-free $g^1_6$ 's; Theorem \ref{t1}.
\end{enumerate}
\end{mainresults} 

The reducibility of $\Hh_{15,14,5}$ is a result of another paper \cite{bumi}, which we listed in Main Results for a more comprehensive list of results concerning $\Hh_{15,g,5}$.
Our main reason for aiming at an extensive study of curves of degree $d=15$  in $\PP^5$ can be addressed as follows.  The case $g=14$, which we studied in \cite{bumi}, was hard enough requiring several new techniques. If we shift our attention to curves of higher genus e.g. $15\le g\le 18=\pi(15,5)$, there appear more interesting cases and components containing fruitful information on the geometry of projective algebraic curves. 

Besides using classical methods in analyzing curves on smooth and singular rational surfaces, we use
the irreducibility of Severi varieties of nodal curves on Hirzebruch surfaces \cite{ty} in a key step in Theorem \ref{main15}.
We also use and prove seemingly well-known facts stating that curves on singular surfaces of low degrees such as singular del Pezzo or cone over an elliptic curve are flat limits of curves on smooth surfaces when we need to rule out the possibility for a family of curves on such singular surfaces constituting a full component; Proposition \ref{b1}, Lemma \ref{se1}. We also use the notion of residual schemes in one of our main results; cf. Remark \ref{residue} and Theorem \ref{t1}.

Unless the genus $g$ of a smooth curve $X\subset\PP^r$ is fairly high with respect to $\deg X$, a curve in a component of $\HO{d,g,r}$ with the very ample  $|\Oo_X(1)|$ is by no means 
contained in a surface of low degrees.
However the residual series $|K_X(-1)|$ may induce possibly singular curves of  lower degrees which may sit on surfaces of small degrees with better classifications in lower dimensional projective spaces. If this occurs, studying (extrinsic) projective curves defined by the residual series may become easier than directly handling the original curve $X\subset\PP^r$.  We use this simple idea several times throughout the paper. 

The organization of this paper is as follows. In the next section, we prepare some preliminaries. In section 3, we treat curves of genus $g\le 13$. In sections 4, 5 we study curves of genus $15$,$16$
and in the final section we finish off  by considering  curves of genus $17, 18$ and making a small observation regarding larger-than-expected components of Hilbert schemes.
\subsection{Notation and conventions}
For notation and conventions, we  follow those in \cite{ACGH} and \cite{ACGH2}; e.g. $\pi (d,r)$ is the maximal possible arithmetic genus of an irreducible,  non-degenerate and reduced curve of degree $d$ in $\PP^r$ which is usually referred to the first Castelnuovo genus bound. We shall refer to  irreducible curves $X\subset \PP^r$ with the maximal possible genus $g=\pi (d,r)$ as {\it extremal curves}. $\pi_1(d,r)$ is the so-called second Castelnouvo genus bound which is the maximal possible arithmetic genus of  an irreducible, non-degenerate and reduced curve of degree $d$ in $\PP^r$ not lying on a  surface of minimal degree $r-1$; cf. \cite[p. 99]{he}, \cite[p. 123]{ACGH}.

Following classical terminology, a linear series of degree $d$ and dimension $r$ on a smooth curve $X$ is denoted by $g^r_d$.
A base-point-free linear series $g^r_d$ ($r\ge 2$) on $X$ is called {\it birationally very ample} when the morphism 
$X \rightarrow \mathbb{P}^r$ induced by  the $g^r_d$ is generically one-to-one onto (or is birational to) its image curve.  A base-point-free linear series $g^r_d$ on $X$  is said to be compounded of an involution ({\it compounded} for short) if the morphism induced by the linear series gives rise to a non-trivial covering map $X\rightarrow C'$ of degree $k\ge 2$. 

We  also recall the following standard set up and notation; cf. \cite[Ch. 21, \S 3, 5, 6, 11, 12]{ACGH2} or \cite[\S 1 and \S 2]{AC2}.
Let $\mathcal{M}_g$ be the moduli space of smooth curves of genus $g$. Given an isomorphism class $[C] \in \mathcal{M}_g$ corresponding to a smooth irreducible curve $C$, there exist a neighborhood $U\subset \mathcal{M}_g$ of  $[C]$ and a smooth connected variety $\mathcal{M}$ which is a finite ramified covering $h:\mathcal{M} \to U$, as well as  varieties $\mathcal{C}$ and $\mathcal{G}^r_d$ proper over $\mathcal{M}$ with the following properties:
\begin{enumerate}
\item[(1)] $\xi:\mathcal{C}\to\mathcal{M}$ is a universal curve, i.e. for every $p\in \mathcal{M}$, $\xi^{-1}(p)$ is a smooth curve of genus $g$ whose isomorphism class is $h(p)$,
\item[(2)] $\mathcal{G}^r_d$ parametrizes the pairs $(p, \mathcal{D})$, where $\mathcal{D}$ is possibly an incomplete linear series of degree $d$ and dimension $r$ on $\xi^{-1}(p)$.
\item[(3)] For any component $\Gg\subset\mathcal{G}^{r}_{d}$, 
\vspace{-2pt}
$$\dim\Gg\ge\lambda(d,g,r):=3g-3+\rho(d,g,r)$$ 
\vspace{-2pt}
and  for any component $\Hh\subset\h{H}_{d,g,r}$,
$$\dim\Hh\ge\h{X}(d,g,r):=\lambda (d,g,r)+\dim\Aut(\PP^r).$$
\end{enumerate}
\vni
Throughout, we work over an algebraically closed field $\KK$ of characteristic zero.
\vni
{\bf Acknowledgement:}
The authors are grateful to the referee for valuable comments and useful suggestions toward an
overall improvement of the readability and clarity of this paper.
\vspace{-8pt}
\section{Preliminaries and several related results}
\subsection{Curves on rational  surfaces and curves on a cone}
We collect  results on curves contained in low degree surfaces. We need the following proposition in the final part of the proof of Theorem \ref{main15}.
\begin{proposition}\label{b1}
Let $S\subset \PP^r$, $3\le r \le 9$, be a normal rational del Pezzo surface of degree $r$. Fix an integral curve $C\subset S$. Then $C$ is a flat limit of a family of curves, 
contained in  smooth del Pezzo surfaces of degree
$r$.
\end{proposition}
\begin{proof}
For $7\le r\le 9$, the proposition is trivial because $S$ is smooth. The case $r=3$ is the main result of \cite{Brevik}. Assume $r>3$. Fix a general $A\subset S$ such that $\#A=r-3$ and let $S'$ be the blowing-up of $S$ at all points of $A$. Let $M\subset \PP^r$ be the $(r-4)$-dimensional linear space spanned by $A$. For a general $A$ the surface $S'$ has the same number of singularities as $S$ and it is isomorphic to the cubic surface $S''\subset \PP^3$ which is the closure in $\PP^3$ of the image of $S\setminus A$ by the linear projection ${\ell} : \PP^r\setminus M\to \PP^3$; \cite[p. 389]{Dolgachev}. Since $A$ is general  we know $C\cap A=\emptyset$. Since the linear projection $\ell$ restricted to $C$ is an embedding, $C\cong \ell(C)$  and $\deg(\ell(C)) =\deg (C)$. 
By \cite{Brevik} there is a smoothing family of $S''$ on which
the curve $\ell(C)$ is a flat limit of a family of curves on nearby smooth cubic surface $S_t$; $t\in T\setminus \{0\}$, $T$ an integral quasi-projective curve, $0\in T$, $S_0= S''$ and $S_t$ smooth for all $t\in T\setminus \{0\}$.  Obviously $\ell(C)\cap \ell_i=\emptyset$ for all  $r-3$ exceptional divisors $\ell _i$ of $S''$ associated to the exceptional divisors of the blowing up of the $r-3$ points of $A$.
By \cite[Proposition 4.7]{Brevik} we may assume that the curve $\ell (C)$ is a limit of curves $C_t\subset S_t$, $t\in T\setminus \{0\}$, which have empty intersection with the $r-3$ exceptional divisors $\ell_i(t)$ of $S_t$, $t\in T\setminus \{0\}$ specializing to the $r-3$ exceptional divisors $\ell _i$ of $S''$.
For $t\in T\setminus \{0\}$, let $\tilde{S}_t$ the surface obtained from $S_t$ by  blowing down the $r-3$ divisors $\ell_i(t)$. Each surface $\tilde{S}_t$ is a smooth del Pezzo surface in $\PP^r$. Since $C_t\cap \ell_i(t)=\emptyset$,
each curve $C_t$ is isomorphic to a curve $\tilde{C}_t\subset \tilde{S}_t$. The flat family over $T$ gives rise to a flat family of smooth del Pezzo surfaces $\tilde{S}_t$, $t\in T\setminus \{0\}$ with $S$ as its flat limit together with a flat family  curves $\tilde{C}_t$, $t\in T\setminus \{0\}$ with $C$ as its flat limit.
\end{proof}
\vspace{-4pt}

When we deal with curves on a cone in $\PP^r$  (usually $r=4,5$) over a rational normal curve, we use the following elementary facts in several places of this paper; in the proofs of Proposition \ref{g=13}, Theorems \ref{main15}, \ref{main16} and Proposition \ref{maing=18}. We include these facts mainly for fixing notation.
\begin{remark}\label{cone1}
\begin{itemize}
\item[(a)] Let $S\subset\PP^r$ ($r\ge 3$)  be a cone over a 
rational normal curve $R\subset H\cong\PP^{r-1}$ with vertex outside $H$.
Recall that $S$ is the image of the birational morphism $\FF_{r-1}=\mathbb{P}(\h{O}_{\PP^1}\oplus\h{O}_{\PP^1}(r-1))\rightarrow S\subset\PP^r$ induced by $|h+(r-1)f|$, where $h^2=-(r-1)$ and
$f$ is a fibre of $\FF_{r-1}\rightarrow\PP^1$. 
Let $C\subset S$ be an integral curve of degree $d$, with  the strict transformation $\w{C}$ of $C$ under  $\mathbb{F}_{r-1}\rightarrow S$. 
~Setting $k=\w{C}\cdot f$, we have $\w{C}\equiv kh+df$ and
\begin{equation}\label{conevertex1}
0\le \w{C}\cdot h=(kh+df)\cdot h=d-(r-1)k=m
\end{equation} where $m$ is the multiplicity of $C$ at the vertex of $S$.
\item[(b)] In the case  (a), there is no smooth $C\subset S\subset\PP^5$  with $\deg C=15$ by \eqref{conevertex1};  i.e. there is no integer $k$ such that
$d=15=m+4k$ with $m\in\{0,1\}$. If $S\subset\PP^5$ is a Veronese surface, there is no irreducible $C\subset S$ with odd $\deg C$.

\item[(c)]  Therefore  we may assume that a quartic  surface $S\subset\PP^5$ containing a smooth curve of degree $d=15$ is a rational normal scroll.

\item[(d)] 
Let $S\subset\PP^r$ be a rational normal surface scroll.
For $X\in|aH+bL|$ -- where $H$ (resp. $L$) is the class of a hyperplane section (resp. the class a line of the ruling) -- we have
\vspace{-3pt}
\begin{equation}\label{sd}
\deg X=(r-1)a+b, ~
p_a(X)=\tbinom{a-1}{2}(r-1)+(r-2+b)(a-1)
\end{equation}
\begin{equation}\label{sdn}
 \dim|aH+bL|=\frac{a \left(a +1\right) (r-1)}{2}+\left(a +1\right) \left(b +1\right)-1
 \end{equation}
 \begin{equation}\label{sf}\dim\h{S}(r)= (r+3)(r-1)-3.  \end{equation}
 where $\h{S}(r)$ is the irreducible family of rational normal surface scrolls in $\PP^r$; \cite[p. 91]{he}.
\end{itemize}
\end{remark}
\subsection{Some remarks on moduli maps}
Let $\mu: \HO{15,g,5}\rightarrow \Mm_g$ denote the natural functorial map - which we call the {\it moduli map} - sending $X\in\HO{15,g,5}$ to its isomorphism class
$\mu (X)\in\h{M}_g$.
In the last two sections, we study $\HO{15,g,5}$ in some detail. For example, we use the following for 
the dimension estimate of the image of the moduli map $\HO{15,18,5}\rightarrow \h{M}_{18}$; Proposition \ref{z6}.
\vspace{-2pt}
\begin{proposition}\label{z5}
Let $Q\subset\PP^3$ be a smooth quadric surface. 
Fix integers $3\le a<b\le 2a-1$ and  a smooth $X\in |\Oo_Q(a,b)|$. Then $X$ is $a$-gonal with a unique $g^1_a$, no base-point-free $g^1_c$ for $a<c<b$ and a unique base-point-free $g^1_b$, the pencil induced by $|\Oo_Q(1,0)|$.
\end{proposition}
\begin{proof}
By \cite[Corollary 1]{Martens}, $X$ is $a$-gonal with a unique pencil $g^1_a$ induced by the pencil $|\Oo_Q(0,1)|$.
By \cite[Theorem 2.1]{paxia} there is no base-point-free $g^1_c$ with $a<c<b$. Take any $g^1_b$ and take a general $A\in g^1_b$. Since $X$ has no base-point-free $g^1_{b-1}$ (or if $b=a+1$ a unique $g^1_a$), the linear series $g^1_b$ is complete.  By adjunction
we have $h^1(Q,\Ii_A(a-2,b-2)) >0$. To prove the lemma  it is sufficient to find $R\in |\Oo_Q(1,0)|$ containing $A$.

Set $A_0:= A$. Take any $L_1\in |\Oo_Q(0,1)|$ such that $e_1:= \#(L_1\cap A_0)$ is maximal and set $A_1:= A_0\setminus A_0\cap L_1$. Take any $L_2\in |\Oo_Q(0,1)|$  such that $e_2:= \#(L_2\cap A_1)$ is maximal and set $A_2:= A_1\setminus A_1\cap L_2$. We define recursively the sets $A_3,\dots,A_b$, the elements $L_3,\dots ,L_b\in |\Oo_Q(0,1)|$ and the integers $e_3,\dots ,e_b$ with $e_i:= \#(A_{i-1}\cap L_i)$ maximal and $A_i:= A_{i-1}\setminus A_{i-1}\cap L_i$. Since at each step we require the maximality of the integer $i$ and $A_{i-1}\supseteq A_i$,
we have $e_i\ge e_{i+1}$ for all $1\le i\le b-1$. Note that if $e_i=0$, then $A_{i-1}=\emptyset$ 
and $A_j=\emptyset$ for all $j>i-1$ with $j\le b$. Thus $A_b=\emptyset$.
\vni
Let $c$ be the first integer such that $\#A_c\le 2$. We saw that $c\le b-2$ and $c=b-2$ if and only if $e_1=1$. Set $T:= L_1\cup \cdots \cup L_c\in |\Oo_Q(0,c)|$. 
Consider the exact sequence
\begin{equation}\label{eqce1}
0\to \Ii_{A_c}(a-2,b-2-c) \to \Ii_A(a-2,b-2)\to \Ii_{T\cap A,T}(a-2,b-2)\to 0.
\end{equation}
First assume $e_1\ge 2$ and hence $c\le b-3$. Since $\Oo_Q(a-2,b-2-c)$ is very ample and $\#A_c\le 2$, we have $h^1(Q,\Ii_{A_c}(a-2,b-2-c))=0$.
Thus the long cohomology exact sequence of \eqref{eqce1} gives $h^1(T,\Ii_{A\cap T}(a-2,b-2))>0$. The set $T$ has $c$ connected components, $L_1,\dots ,L_c$, with $L_i\cong \PP^1$. Since 
$
\displaystyle{h^1(T,\Ii_{A\cap T,T}(a-2,b-2))= \sum_{i=1}^{c} h^1(L_i,\Ii_{A\cap T,T}(a-2,b-2)_{|L_i})>0,}
$
we get $h^1(L_i,\Ii_{A\cap T,T}(a-2,b-2)_{|L_i})>0$ and $e_i\ge a$ for some $i$; by $\deg \Oo_{L_i}(a-2,b-2) =a-2$ for all $i$, we have $h^1(L_i,\Ii_{A\cap T,T}(a-2,b-2)_{|L_i})=0$ if $e_i\le a-1$ for all $i$.
 But then $A$ contains a subset $A'$ such that $\#A'=a$, $A'\subset L_i$ and $|A'|$ is the $g^1_a$ on $X$.
Since the $g^1_b$ is base-point-free and complete, we get a contradiction.

Now assume $e_1=1$. Thus $c=b-2$ and $\#(A\cap L)\le 1$ for all $L\in |\Oo_Q(0,1)|$. Thus $h^1(T,\Ii_{T\cap A, T}(a-2,b-2)) =0$. Hence the long cohomology exact sequence
of \eqref{eqce1} gives $h^1(Q,\Ii_{A_c}(a-2,0))>0$, i.e. two points in $A_c=\{p,s\}$ fails to impose independent conditions on 
$|\Oo_Q(a-2,0)|$. Therefore there is  $R\in |\Oo_Q(1,0)|$ such that $A_c\subset R$. 
Fix any $q\in A\setminus A_c$.
Since $\#(A\cap L)\le 1$ for all $L\in |\Oo_Q(0,1)|$, there are $D_1,\dots,D_{b-2}\in |\Oo_Q(0,1)|$ such that $T':= D_1\cup \dots \cup D_{b-2}$ contains $A\setminus \{p,q\}$.
Using $T'$ instead of $T$, we get the existence of $R'\in |\Oo_Q(1,0)|$ such that $\{p,q\}\subset R'$. Since $R$ is the unique element of $|\Oo_Q(1,0)|$ containing $p$, $R'=R$.
Thus $A\subset R$.
\end{proof}
The following is a consequence of Proposition \ref{z5} which we use in \textsection{6.1}; Propostion \ref{z6}.
\begin{corollary}\label{z5.010}
Fix integers $2a-1\ge b>a\ge 3$ and smooth $X,\tilde{X}\in |\Oo_Q(a,b)|$. If $\tilde{X}\cong X$ as abstract curves, then there is $v\in \Aut(\PP^1)\times \Aut(\PP^1)$ such that $v({X})={\tilde{X}}$.
\end{corollary}
\begin{proof}
Take any isomorphism $\tilde{X}\stackrel{u}{\cong}{X}$.  The line bundles $R_1:= u^\ast (\Oo_X(0,1))$ and $R_2:= u^\ast (\Oo_X(1,0))$ are the unique base-point-free line bundles on $\tilde{X}$ of degrees $a$ and $b$, respectively by  Proposition \ref{z5}, and hence  $R_1= \Oo_{\tilde{X}}(0,1)$ and $R_2 =\Oo_{\tilde{X}}(1,0)$.
Thus $u$ induces  isomorphisms $|\Oo_{{X}}(1,0)|=\PP^1\to \PP^1=|\Oo_{\tilde{X}}(1,0)|$ and $|\Oo_{{X}}(0,1)| \to |\Oo_{\tilde{X}}(0,1)|$ and the corresponding pair in $\Aut(\PP^1)\times \Aut(\PP^1)$  induces
$v$.\end{proof}

We use the following remark,
similar to Corollary \ref{z5.010}, for the study of the moduli map $\mu: \HO{15,16,5}\rightarrow \Mm_{16}$ in \textsection {5.2}; Propositions \ref{mz5}, \ref{mz6}.
\begin{remark}\label{z5.011}
Let ${S}$ and $\tilde{S}$ be rational normal surface scrolls in $\PP^5$. If ${S}\stackrel{u}{\cong} \tilde{S}$ as abstract varieties, then they are projectively equivalent. To see this,  assume $S\cong\w{S}\cong \FF_2$. Note that the only very ample line bundle on $\FF_2$ inducing an embedding onto a surface of minimal degree $S\cong\tilde{S}\subset\PP^5$ is $\Oo_{\FF_2}({h+3f})$. 
Hence  there is  $v\in \Aut( \PP H^0(\FF_2, \Oo_{\FF_2}({h+3f})))=\Aut(\PP^5)$ such that $v(S)=\tilde{S}$ and $v_{|S} =u$.
Let $S\cong\tilde{S}\cong\FF_0\cong\PP^1\times\PP^1$. Assume that  $u\in \Aut(\PP^1)\times \Aut(\PP^1)$. The only very ample line bundle on $\FF_0$ inducing an embedding 
onto a minimal degree surface in $\PP^5$ is $\Oo_{\PP^1\times\PP^1}(1,2)$, hence the conclusion follows.
\end{remark}
\vspace{-8pt}
\section{Curves of genus $g\le 13$ and degree $d=15$ in $\PP^5$}
We denote by $\HL{d,g,r}$ the subscheme of $\h{H}_{d,g,r}$ consisting of components of $\h{H}_{d,g,r}$ whose general element is linearly normal. We set $\alpha :=g-d+r=g-10$ for $(d,r)=(15,5)$. 
For a general element $X$ in a component of $\HL{15,g,5}$,
$\alpha$ is the index of speciality of $|\Oo_X(1)|$.
the complete hyperplane series $\h{D}$ of $X\subset\PP^5$. 
For a possible component $\h{H}\subset\HO{15,g,5}$ which may  not be a component of  $\HL{15,g,5}$,
we set $\beta:=g-15+\dim |\h{D}| >g-10=\alpha$, where $\h{D}$ is the (incomplete) hyperplane series of a general $X\in\h{H}$.

In this section we study $\HO{15,g,5}$ for genus $g\le 13$, which is relatively simple to handle. 
We recall the following concerning an upper bound of the dimension of a birationally very ample linear series on curves, which enables us to simplify some of our computation. 
\begin{remark}\label{bound}
By \cite[p.75]{he}, the largest possible dimension of a birationally very ample linear series of degree $d\ge g$ on a curve of genus $g$ is $\frac{2d-g+1}{3}$.
\end{remark}
\begin{proposition}\label{g=13} $\HO{15,13,5}$ is reducible with two components $\h{H}_1$ and $\h{H}_2$.
\begin{itemize}
\item[(i)] $\h{H}_1=\HL{15,13,5}$ and $\dim \HL{15,13,5}=\h{X}(d,g,r)=66$ with a $7$-gonal general element.
\item[(ii)] $\dim \h{H}_2=68>\h{X}(d,g,r),$ a general $X\in\h{H}_2$ is trigonal consisting of the image of external projection of extremal curves of degree $15$  in $\PP^6$.
\end{itemize}
\end{proposition}
\begin{proof}
We denote by $\Sigma_{d,g}$ the {\it irreducible} Severi variety of plane curves of degree $d$ and genus $g$; cf.\cite {H2}. By \cite[Theorem 3.7]{lengthy} \& \cite[Theorem 2.5]{JPAA}, $\HL{15,13,5}\neq\emptyset$ and is irreducible of dimension $\Xx(15,13,5)$. For a general $X\in\HL{15,13,5}$, $|K_X(-1)|=g^2_9$ is birationally very ample, base-point-free and hence $X$  has a plane model of degree $9$ which follows from the proof of \cite[Theorem 2.5]{JPAA}. Therefore $X\in\HL{15,13,5}$ corresponds to an element of $\Sigma_{9,13}$. 

Conversely, the residual series of the series cut out by lines in $\PP^2$ on (the non-singular model of) a general member in $\Sigma_{9,13}$ is a very ample $g^5_{15}$ by a result of  d'Almeida and Hirschowitz \cite[Theorem 0]{Coppens}; cf. \cite[pp 13--15]{lengthy} for details in a similar situation.
Therefore we have a generically one-to-one correspondence $$\HL{15,13,5}/\Aut(\PP^5)\stackrel{bir}{\cong} \Sigma_{9,13}/\Aut(\PP^2)$$
via residualization. 
Recall that a general member in $\Sigma_{9,13}$ is a plane curve of degree $9$ with $\delta=15$ nodes as its only singularities.  By a theorem of M. Coppens \cite[Theorem]{Coppens0}, $X$ is $7$-gonal with gonality pencils cut out by lines through a node on the plane model of degree $9$. 

Suppose there is a component $\h{H}$ 
other than $\HL{15,14,5}$. A general $X\in\h{H}$ is not linearly normal, $\beta =4$ by Remark \ref{bound} and we have
$|\h{D}|=g^6_{15}$ where $\Dd$ is the incomplete hyperplane series of $X\subset\PP^5$.
Since $\pi (15,6)=13=g$, the curve $\w{X}\subset\PP^6$ induced by the complete $|\h{D}|=g^6_{15}$ is an extremal curve lying on a quintic surface $S\subset\PP^6$. If $S$ is a smooth rational normal scroll, by solving \eqref{sd},  we get $\w{X}\in |3H|$, $|K_S+\w{X}-H|=|3L|$ and $X\cong\w{X}$ is trigonal with the trigonal pencil cut out by the ruling $|L|$. If $S$ is a cone over a rational normal curve in $\PP^5$, by \eqref{conevertex1} in Remark \ref{cone1} (a), we again may conclude that $X$ is trigonal. Conversely, on a trigonal curve $X$ of genus $g=13$, $|K_X-3g^1_3|=g^6_{15}$ is very ample, hence the moduli map $\h{H} \rightarrow \h{M}^1_{g,3}$ is dominant. The incomplete very ample linear series $g^5_{15}$ which is a codimension 
one subspace of the complete $|K_X-3g^1_3|$ over the family of trigonal curves $\h{M}^1_{g,3}$ form an irreducible family $\h{F}\subset\Gg^5_{15}$ of dimension
$$\dim\h{M}^1_{g,3}+\dim\mathbb{G}(5,6)=33>3g-3+\rho(d,g,r)=31.$$
Therefore the $\Aut(\PP^5)$-bundle over $\h{F}$ may contribute to an extra component $\h{H}_2$ other than $\HL{d,g,r}$. 
By lower semicontinuity of gonality, $\h{H}_1$ is not in the boundary of $\h{H}_2$. Hence $\h{H}_1$ and $\h{H}_2$
are the only two distinct components of $\HO{15,13,5}$ by exhaustion.
\end{proof}
The following lemma is easy to prove and useful when we deal with double coverings of  curves of  small genus which may be induced by the compounded residual series $|K_X(-1)|$, $X\in \HO{d,g,r}$; Proposition \ref{main12} and Theorem \ref{main15}.

\begin{lemma} \label{bih}Let $X\stackrel{\eta}{\rightarrow} E$ be a double covering of a curve $E$ of genus $h\ge 1$. Let $\h{E}=g^s_{e}$ be a non-special linear series on $E$. Assume that $|\eta^*(g^s_{e})|=g^s_{2e}$. Then the base-point-free part of the complete  
$|K_X(-\eta^*(g^s_{e}))|$ is compounded.
\end{lemma}
\begin{proof} Note that for any $p\in E$,  $|g^s_{e}+p|=g^{s+1}_{e+1}$ since $g^s_{e}$ is non-special and
 $$\dim|\eta^*(g^s_{e})+\eta^*(p))|=\dim|\eta^*(g^s_{e}+p)|=\dim|\eta^*(g^{s+1}_{e+1})|\ge \dim\h{E} +1.$$
 Hence $|K_X(-\eta^*(g^s_{e}))|$  is compounded. \end{proof}

\begin{proposition}\label{main12} $\HO{15,12,5}=\HL{15,12,5}$ is irreducible of the expected dimension and a general $X\in\HO{15,12,5}$ is $7$-gonal.
\end{proposition} 
\begin{proof} 
By \cite[Theorem 2.3]{JPAA},
$\HL{15,12,5}$ is irreducible. Suppose there is a component $\h{H}\neq\HL{15,12,5}$. For a general $X\in\h{H}$ with an incomplete hyperplane series $\h{D}$,  $\beta=g-15+\dim |\h{D}|\ge 3$ and hence $\beta =3$ by Remark \ref{bound}. 

If $\h{E}=|K_X-\h{D}|=g^{\beta-1}_7=g^2_7$ is compounded, $\h{E}$ has non-empty base locus $\Delta$ with $\deg\Delta=1$.  So $X$ is bielliptic which is impossible by Lemma \ref{bih}.

If $\h{E}$ is birationally very ample, $\h{E}$ is base-point-free since $g=12> \binom{6-1}{2}$ and hence $X$ has a plane model of degree $7$. Consider the Severi variety $\Sigma_{7,12}$ of plane septics of genus $g=12$ whose general member is a curve with $\delta =3$ nodes. 
Since
$$\dim\Sigma_{7,12}=\dim|\Oo_{\PP^2}(7)|-\delta=3\cdot 7+g-1=32,$$ the family $\h{F}\subset\h{G}^2_7$ consisting 
of base-point-free birationally very ample nets of degree $7$ on moving curves has dimension  $$\dim\h{F}=\dim\Sigma_{7,12}-\dim\Aut{(\PP^2)}=24.$$
Hence the residual family $\h{F}^\vee:=\{|K_X-\h{E}|; \h{E}\in\h{F}\}\subset\h{G}^6_{15}$ has dimension $24$ and 
the family $\w{\h{F}}$ consisting  of incomplete very ample $g^5_{15}$'s arising this way has dimension 
$$\dim\h{F}^\vee+\dim\mathbb{G}(5,6)=30< 3g-3+\rho(15,12,5)=33.$$
Thus $\w{\h{F}}$ does not constitute a full component.

Note that $\alpha=2$, $\rho (d,g,5)=g-6\alpha=0$ and hence there is a unique component
of $\HO{d,g,r}$ dominating $\h{M}_g$ (\cite[Theorem, p. 69-70]{he}), which is $\HL{d,g,r}.$ 
A general curve of genus $g=12$ is $7$-gonal by the Brill-Noether theorem. 
\end{proof}

\begin{proposition}\label{main11} $\HO{15,11,5}=\HL{15,11,5}$ is irreducible of the expected dimension and 
a general $X$ is $7$-gonal. 
\end{proposition} 
\begin{proof} We have $\alpha=1$.
By \cite[Theorem 2.2]{JPAA},
$\HL{15,11,5}$ is irreducible of the expected dimension. Suppose there is a component $\h{H}$ 
other than $\HL{15,11,5}$. A general $X\in\h{H}$ is not linearly normal,  $\beta\ge 2$ and hence $\beta =2$ by Remark \ref{bound}.  Thus we have $|K_X-\h{D}|=g^1_5$ and $|\h{D}|=g^6_{15}$. A general element of $\h{H}$ is induced by $6=\dim\mathbb{G}(5,6)$ dimensional 
subseries of a complete $g^6_{15}=|K_X-g^1_5|$. Such  family of incomplete 
$g^5_{15}$'s forms an irreducible family of dimension at most
$$\dim\h{M}^1_{g,5}+\dim\mathbb{G}(5,6)=33<3g-3+\rho(15,11,5)=35, $$
hence such family does not contribute to a full component of $\HO{15,11,5}$.

Since $\rho (d,g,5)>0$, the unique component $\HL{d,g,r}$ of $\HO{d,g,r}$ dominates $\Mm_g$
whose general curve member is $7=[\frac{g+3}{2}]$-gonal. 
\end{proof}

\begin{proposition}\label{10less}For $g\le 10$, $\HO{15,g,5}$ is irreducible of the expected dimension and a general element $X$ 
is $[\frac{g+3}{2}]$-gonal.
\end{proposition}
\begin{proof} This follows from  \cite[p.75]{he} for $g\le 9$ and \cite[Theorem]{PAMS} for $g=10$. 
\end{proof}

\section {Curves of genus $g=15$ }
This section is devoted to the family of curves of genus $g=15$. The following subsection contains  the main result of this section.
\subsection{Irreducibility of $\HO{15,15,5}$}
\begin{theorem}\label{main15} $\HL{15,15,5}=\h{H}_{15,15,5}$ is irreducible whose general element
is a $5$-gonal curve lying on a smooth del Pezzo surface in $\PP^5$ and  $\dim\h{H}_{15,15,5}=64$. 
\end{theorem}
\begin{proof} Note that a general element in any component of $\h{H}_{15,15,5}$ is linearly normal since $\pi (15,6)=13<g=15$ and hence $\h{H}_{15,15,5}=\HL{15,15,5}$.

Since $\pi_1(15,5)=16>g$, $X$ is not extremal or {\it nearly extremal} -- a curve 
with $p_a(X)> \pi_1(d,r)$ -- and hence $X\subset\PP^5$ may not sit on a surface of small degree.  Instead, we look at the residual series $|K_X(-1)|$ and study the curve induced by it. This technique -- if one may call this ``a technique" --  is not new and  has been used before, e.g. in \cite{lengthy} or possibly in works by other authors.

For a general $X\in\h{H}_{15,15,5}$, set $\h{E}:=g^4_{13}=|K_X(-1)|$ and let 
$\ce\subset\PP^4$ be the dual curve of $X$, which is by definition the image curve induced by the base-point-free part of $\h{E}$. We first claim that 
$\h{E}$ is birationally very ample, possibly with non-empty base locus.

\noindent
\vni
{\bf Claim 1:} $|K_X(-1)|$ is birationally very ample.

\noindent
Suppose
$|K_X(-1)|$ is compounded. Since $\deg |K_X(-1)|=13$ is prime, $|K_X(-1)|$ has non-empty base locus $\Delta$ and we have the following four possibilities: 
\begin{itemize}
\item[(i)]
$|K_X(-1)|=g^4_{12}+\Delta$, $\deg\Delta=1$, 
\item[(ii)] $|K_X(-1)|=g^4_{10}+\Delta$, $\deg\Delta=3$,
\item[(iii)] $|K_X(-1)|=g^4_{9}+\Delta$, $\deg\Delta=4$; this can be excluded since the dual curve  $\ce\subset\PP^4$ induced by the compounded $g^4_{9}$ is non-degenerate. 

\item[(iv)] $|K_X(-1)|=g^4_{8}+\Delta$, $\deg\Delta=5$; in this case, $X$ is hyperelliptic and does not carry a very ample special linear series. 
\end{itemize}

\noindent
(i) Suppose $|K_X(-1)|=g^4_{12}+\Delta$, $\deg\Delta=1$, then one of the following holds;
\[
\begin{cases}
\mathrm{(ia)} ~X \mathrm{ ~is ~~trigonal ~~with~}  |K_X(-1)(-\Delta)|=g^4_{12}=4g^1_3 \mathrm{ ~or~~ }
\\
\mathrm{(ib)} ~\exists ~~X\stackrel{\phi}{\rightarrow}\ce\subset\PP^4, \deg\phi=2,

 \mathrm{~~genus}(\ce)=2, g^4_{12}=
\phi^*(g^4_6).
\end{cases}
\]

(ia)
Set $\Delta=p$ and let $p+q+r\in g^1_3$. We have
$$|K_X-(4g^1_3+p)|=|K_X-5g^1_3+q+r|$$ and hence
\begin{align*}
|\h{O}_X(1)-q-r|&=|K_X-K_X(-1)-q-r|\\&=|K_X-(4g^1_3+p)-q-r|=|K_X-5g^1_3|=g^{s}_{13}, s\ge 4
\end{align*}
from which it follows that  $|\mathcal{O}_X(1)|$ is not very ample.

(ib) Since $\deg \phi=2$, $\deg \phi (X)=\deg\ce=6$ and $g(\ce)\le\pi(6,4)\le 2$ by Castelnuovo genus bound.
Since $|K_X(-1)(-\Delta)|$ is complete, $g^4_6$  on the normalization of $\ce$ such that  $g^4_{12}=|K_X(-1)(-\Delta)|=
\phi^*(g^4_6)$ is also complete. Hence $\ce$ has geometric genus $g(\ce)=2$, smooth and $|\h{O}_{\ce}(1)|=g^4_6$ is non-special.
However we may exclude this case by Lemma \ref{bih}.
 
\vni
(ii) Suppose $|K_X(-1)|=g^4_{10}+\Delta$, $\deg\Delta=3$. Since $\ce\subset\PP^4$ is non-degenerate and $X$ is non-hyperelliptic, we have 
$g^4_{10}=\phi^*(g^4_5)$ where $\phi :X\rightarrow \ce$ is a double cover of an elliptic curve $\ce$
 with complete,  non-special $g^4_5$.  Again by  Lemma \ref{bih}, we may exclude this case, finishing the proof of Claim 1. 

\vskip 4pt
\noindent
{\bf Claim 2:} No smooth $X\in\h{H}_{15,15,5}$ is trigonal: Suppose there is  a trigonal $X\in\h{H}_{15,15,5}$. 
Recall that on a trigonal curve $X$ of genus $g\ge 5$ with a $g^r_d, ~ d\le g-1$, either $g^r_d$ is compounded with the unique $g^1_3$  or  $|K_X-g^r_d|$ is compounded with $g^1_3$ by Maroni theory; cf. \cite[Proposition 1]{ms}. Since $\h{E}$ is birationally very ample by Claim 1, $|K_X-\h{E}|=|\Oo_X(1)|$ is  compounded with $g^1_3$, a contradiction.   This finishes the proof of Claim 2.
\vni
{\bf Claim 3:} $|K_X(-1)|$ is base-point-free.

For the birationally very ample $\Ee=|K_X(-1)|$, suppose $\Delta=\mathrm{Bs}(\h{E})\neq\emptyset$ and set $\h{E}=\w{\h{E}}+\Delta$. Note that 
$\deg\Delta=1$ otherwise $\pi(e,4)\le 12<g$ if $e\le 11$. Put $\Delta=p$ and let $C_{\w{\h{E}}}\subset\PP^4$ be the image of the morphism induced by the moving part $\w{\h{E}}$. Since $\deg C_{\w{\h{E}}}=12$ and $\pi(12,4)=15=g$,  $C_{\w{\h{E}}}$
is an (smooth) extremal curve lying on a cubic surface $S\subset\PP^4$.

\vskip 1.5pt
\noindent
(i) Suppose $S$ is a smooth cubic scroll in $\PP^4$ and let $C_{\w{\h{E}}}\in|aH+bL|$. We solve the degree and genus formula \eqref{sd} for $r=4$
to get $a=4, b=0$. Since $X\cong C_{\w{\h{E}}}$  and $|K_X(-1)|=\h{E}=\w{\h{E}}+\Delta$ we have
$$|K_S+C_{\w{\h{E}}}-H|_{|X}=|H+L|_{|X}=|K_X-\w{\h{E}}|,$$ which is birationally very
ample.
From the standard exact sequence
$$0\rightarrow \h{I}_{C_{\w{\h{E}}}}(H+L)\rightarrow \h{O}(H+L)\rightarrow\h{O}(H+L)\otimes\h{O}_{C_{\w{\h{E}}}}\rightarrow 0$$
the restriction map 
$H^0(S, \h{O}(H+L))\stackrel{\rho}{\longrightarrow} H^0(C_{\w{\h{E}}},\h{O}(H+L)\otimes\h{O}_{C_{\w{\h{E}}}})$ is injective since $H^0(S,  \h{I}_{C_{{\w{\h{E}}}}}(H+L))=H^0(S,  \h{O}(-3H+L))=0$. $\rho$ is surjective as well
by Castelnuovo genus bound; if $\rho$ is not surjective  then   $$\mathbb{C}^7\cong H^0(S, \h{O}(H+L))\subsetneq
H^0(C_{\w{\h{E}}},\h{O}(H+L)\otimes\h{O}_{C_{\w{\h{E}}}}),$$ which would imply that  the birationally very ample
$|\h{O}(H+L)\otimes\h{O}_{C_{\w{\h{E}}}}|=|K_X-\w{\h{E}}|$ on $C_{\w{\h{E}}}\cong X$ induces a morphism $X\rightarrow\PP^s, s\ge 7$,  a contradiction by $\pi(16,7)=12<g$.
 
 By the surjectivity of the restriction map $\rho$ and by the very ampleness of $|H+L|$ on the cubic scroll $S$, $|K_X-\w{\h{E}}|$ is  very ample. Note that $|\h{O}_X(1)|=|K_X-\h{E}|=|K_X-\w{\h{E}}-\Delta|$ gives rise to the morphism
 which is the composition of the embedding $X\stackrel{\phi}{\longrightarrow}\PP^6$ given by the very ample $|K_X-\w{\h{E}}|$ followed by the projection $\tau$ with center at $p$; 
 \[
  \begin{tikzcd}
   {X\cong C_{\w{\h{E}}}}  \arrow{r}{|K_X-\w{\h{E}}|} \arrow[swap]{dr}{|K_X-\w{\h{E}}-\Delta|} & \phi(C_{\w{\h{E}}})\subset\PP^6 \arrow{d}{\tau} \\
     & X\subset\PP^5
  \end{tikzcd}
\]
However, the $4$-secant line through $\Delta$, i.e. the line though $\Delta$ in the ruling $|L|$  produces a singularity hence $|\h{O}_X(1)|$ is not very ample if $\Delta\neq \emptyset$.

\vni 
(ii) Suppose $S\subset\PP^4$ is a cone over a twisted cubic in $\PP^3$. 
Let $\w{C_{\w{\h{E}}}}\subset \mathbb{F}_3$ be the strict transformation of $C_{\w{\h{E}}}\subset\PP^4$ under 
$\mathbb{F}_{3}\stackrel{|h+3f|}{\longrightarrow} S$ and set $\w{C_{\w{\h{E}}}}\in |ah+bf|$. Recall that 
$C_{\w{\h{E}}}$ is smooth (extremal) and so is $\w{C_{\w{\h{E}}}}\cong X$. From
$$\w{C_{\w{\h{E}}}}\cdot (h+3f)=(ah+bf)\cdot (h+3f)=b=\deg\w{\Ee}=12$$
$$\w{C_{\w{\h{E}}}}\cdot(\w{C_{\w{\h{E}}}}+K_{\mathbb{F}_3})=(ah+bf)\cdot((a-2)h+(b-5)f)=2g-2=28$$
we get $a=4$ and  $\w{C_{\w{\h{E}}}}\in |4h+12f|$.
Note that
\[
 |K_{\mathbb{F}_3}+\w{C_{\w{\h{E}}}}-(h+3f)|=|h+4f|
 \]
is very ample by \cite[V.Cor. 2.18, p. 380]{Hartshorne}.  We consider the restriction map
$H^0(\mathbb{F}_3, \h{O}(h+4f))\stackrel{\rho}{\longrightarrow} H^0(\w{C_{\w{\h{E}}}},  \h{O}(h+4f)\otimes\h{O}_{\w{C_{\w{\h{E}}}}})=H^0(\w{C_{\w{\h{E}}}},  |K_X-\w{\h{E}}|).$ By the same routine as we did for the previous case (smooth cubic scroll case), we may claim that the restriction map $\rho$ is  surjective. 
 
 By the surjectivity of the restriction map $\rho$ and  the very ampleness of $|h+4f|$ on $\FF_3$, we see that  $|K_X-\w{\h{E}}|$ is very ample. 
 $|\h{O}_X(1)|=|K_X-\w{\h{E}}-\Delta|$ induces the composition of two maps 
$X\stackrel{|h+4f|_{|X}}{\longrightarrow}\PP^6\stackrel{\pi_p}{\longrightarrow}\PP^5$  and the projection map $\pi_p$ is not an embedding; since $f\cdot (h+4f)=1, f\cdot \w{C_{\w{\h{E}}}}=4$,  the image of $f$  under $|h+4f|$  in $\PP^6$ is a $4$-secant line to the smooth image of
$X\cong\w{C_{\w{\h{E}}}}$. Therefore $|\h{O}_X(1)|$ is not very ample, {\it finishing the proof
of Claim 3.}
 \vni
 
Since $\h{E}=|K_X(-1)|$ is base-point-free and birationally very ample, we have $$g=\pi_1(13,4)=15\le p_a(\ce)\le\pi(13,4)=18.$$
Thus $\ce\subset\PP^4$ lies on a surface $S$, $3\le\deg S\le4$; cf. \cite[Theorem 3.15, p.99]{he}.

\item[(A)] Suppose $\deg S=3$ and $S$ is smooth.  Let $\ce\in|aH+bL|$. Solving  \eqref{sd}, i.e.
$13=\ce\cdot H=3a+b, ~p_a(\ce)=\frac{3(a-1)(a-2)}{2}+(2+b)(a-1),$
the following pairs $(a,b)\in\mathbb{Z}\times\mathbb{Z}$ are possible for  $15\le p_a(\ce)\le 18$:
\[
\begin{cases}
p_a(\ce)=18: (5,-2), (4,1), \textrm{cases (A1), (A2) below} \\
p_a(\ce)=17, 16: \mathrm{no ~~ solution }\\
p_a(\ce)=15: (3,4), (6-5), \textrm{cases (A3), (A4) below} \
\end{cases}
\]
Before proceeding, we recall some standard notation concerning linear systems and divisors on a blown up projective plane. Let $\PP^2_s$ be the rational surface $\PP^2$ blown up at $s$ general points. Let $e_i$ be the class of the exceptional divisor
$E_i$ and $l$ be the class of a line $L$ in $\PP^2$. For integers  $b_1\ge b_2\ge\cdots\ge b_s$, let $(a;b_1,\cdots, b_i, \cdots,b_s)$ denote class of the linear system $|aL-\sum b_i E_i|$ on $\PP^2_s$.  By abuse of notation we use the  expression $(a;b_1,\cdots, b_i, \cdots,b_s)$ for the divisor $aL-\sum b_i E_i$ and $|(a;b_1,\cdots, b_i, \cdots,b_s)|$ for the linear system $|aL-\sum b_i E_i|$. We use the convention 
\vspace{-3pt}
$$(a;b_1^{s_1},\cdots,b_j^{s_j},\cdots,b_t^{s_t}), ~ \sum s_j=s$$ 
when  $b_j$ appears $s_j$ times consecutively  in the linear system $|aL-\sum b_i E_i|$.

For computational reasons, we sometimes identify a smooth rational normal surface scroll $S\subset \PP^4$ with the Hirzebruch surface $\mathbb{F}_1$ embedded by the very ample linear system $|e+2f|$ on $\PP^2_1$. We also make obvious identifications among divisor classes such as; 
\vspace{-2pt}
$$H\cong e+2f, L\cong f, l-e_1\cong f_1$$ where $e$ is the class of the minimal degree
self-intersection curve, $f$ is class of the fibre on $\mathbb{F}_1$ ($l$ is the class of a line, $e_1$ the exceptional divisor on $\PP^1_1$ and $f_1$ is the proper transformation of the line through the blown up point). By abusing notation we make no 
distinction between $f$ and $f_1$ ($e$ and $e_1$) and use the same letters.

\vskip 2pt
\noindent
\item[(A1)] $\ce\in |5H-2L|$, $p_a(\ce)=18$: 
Note that 
\begin{equation}\label{octic}\ce\in |5H-2L|=|5(2l-e_1)-2(l-e_1)|=|8l-3e_1|
\end{equation} and hence $\ce$ has a plane model of degree $8$. On a fixed $S\cong\FF_1$, we consider the Severi variety $\Sigma_{{15},\h{M}}$ of curves
of genus $g=15$ in the linear system $\h{M}=|5H-2L|$.  It is known that $\Sigma_{{15},\h{M}}$ is {\it irreducible}, 
\vspace{-2pt}
$$\dim\Sigma_{{15},\h{M}}=\dim\h{M}-\delta , ~~\delta=p_a(C)-g$$  and a general element of $\Sigma_{{15},\h{M}}$ is a nodal curve with $\delta =3$ nodes as its only singularities; cf. \cite{ty}.
We now take a general $C\in \Sigma_{{15},\h{M}}$, i.e.  a curve with three nodes in $|5H-2L|$ on $S$, or equivalently  a curve with a plane model of degree $8$ with an ordinary triple point and three nodes.
Let $\w{C}\subset\PP^2_4$ be the strict transformation of $C$ under the blowing up $\PP^2_4 
\stackrel{\phi}{\longrightarrow} 
S\cong\PP^2_1$ at three nodal singularities. By abusing notation, we set  $H:=\phi^*(\Oo_S(1))$. We have $\w{C}\in |(8;3,2^3)|$ and  
\vspace{-2pt}
\begin{equation}\label{del}
|K_{\PP^2_4}+\w{C}-(2l-e_1)|=|-(3;1^4)+(8;3,2^3)-(2;1,0^3)|=|(3;1^4)|.
\end{equation}
On the other hand, the restriction map
 $$H^0(\PP^2_4,\h{O}(K_{\PP^2_4}+\w{C}-H)\longrightarrow H^0(\w{C}, \h{O}(K_{\PP^2_4}+\w{C}-H)\otimes\h{O}_{\w{C}})$$ is an isomorphism; injective by $H^0(\PP^2_4, \h{O}(K_{\PP^2_4}-H))=0$  and surjective by the Castelnuovo genus bound $\pi(15,6)=13<g$, as we did in the proof of Claim 3.
Therefore the residual series  
$$|K_{\PP^2_4}+\w{C}-\phi^*\Oo_S(1)|_{|\w{C}}=|K_{\w{C}}-\phi^*\Oo_S(1)_{|\w{C}}|$$
of $|\phi^*\Oo_S(1)|_{|\w{C}}$ is completely cut out by the very ample linear system $|(3;1^4)|$ on $\PP^2_4$ which induces an embedding $\w{C}\to\PP^5$ as a smooth curve of degree $$(3;1^4)\cdot \w{C}=(3;1^4)\cdot (8;3,2^3)=24-3-6=15.$$

Let $\h{F}\subset\h{G}^4_{13}$ be the family of such $\h{E}=g^4_{13}$'s arising this way; i.e. the family of complete linear series $\h{E}=g^4_{13}$'s such that $\ce\in \Sigma_{15,\h{M}}$ on a  rational normal scroll $S\subset\PP^4$, $\h{M}=|5H-2L|$ and $\h{E}=|\phi^*\Oo_S(1)|_{|\w{\ce}}$. $\h{F}$ is irreducible since $\Sigma_{15,\h{M}}$ is irreducible.  By an easy dimension count, 
\begin{align*}
\dim\h{F}&=\dim\Sigma_{{15},\h{M}}-\dim\mathrm{Aut}(S)\\&=\dim|5H-2L|-\delta-\dim\mathrm{Aut}(S)=\dim|(8;3)|-3-6=29\\&>3g-3+\rho(13,15,4)=3g-3+\rho(15,15,5)=27,
\end{align*}
hence  {\it the family $\h{F}$ is our first candidate which may contribute to a full component of $\HO{15,15,5}$}. Explicitly,
we have
the natural map $\h{F}\stackrel{\psi}{\rightarrow} \h{F}^\vee$ sending $\h{E}=g^4_{13}$ to its residual series $|K-\h{E}|=g^5_{15}$ such that $\psi(\h{E})$ is very ample for a general $\h{E}\in\h{F}$. Thus, there exists an irreducible family of 
smooth curves of degree $d=15$ and $g=15$ in $\PP^5$ which is an $\Aut(\PP^5)$-bundle over the family $\h{F}^\vee\subset\h{G}^5_{15}$.

\item [(A2)] $\ce\in |4H+L|$, $p_a(\ce)=18$:
We have $\ce\in|4H+L|=|4(2l-e_1)+l-e_1|=|9l-5e_1|$.
Set $\h{L}=|4H+L|$ and let  $\Sigma_{{15},\h{L}}$ be the Severi variety $\Sigma_{{15},\h{L}}$ of curves
of genus $g=15$ in the linear system $\h{L}$.  A general element of $\Sigma_{{15},\h{L}}$ is a nodal curve with $\delta =3$ nodes as its only singularities.
We take a general $C\in \Sigma_{{15},\h{L}}$.
Let $\w{C}\subset\PP^2_4$ be the strict transformation of $C$ under the blow up of $S\cong\PP^2_1$ at three nodal singularities. Hence $\w{C}\in|(9;5,2^3)|$ and  
$$|K_{\PP^2_4}+\w{C}-(2l-e_1)|=|-(3;1^4)+(9;5,2^3)-(2;1,0^3)|=|(4;3,1^3)|.$$
The restriction map $$H^0(\PP^2_4,\h{O}(K_{\PP^2_4}+\w{C}-H)\longrightarrow H^0(\w{C}, \h{O}(K_{\PP^2_4}+\w{C}-H)\otimes\h{O}_{\w{C}})$$ is an isomorphism; injective by  $H^0(\PP^2_4, \h{O}(K_{\PP^2_4}-H))=0$ and surjective by the Castelnuovo genus bound $\pi(15,6)=13<g$ by the same routine as we did in (i) in the proof of Claim 3.
Therefore the residual series of $|\phi^*\Oo_S(1)|_{|\w{C}}$ is completely cut out by the linear system $|(4;3,1^3)|$ on $\PP^2_4$. Note that $|(4;3,1^3)|$ is not very ample by \cite{sandra}. Furthermore we have $(9;5,2^3)\cdot (2,1,1,0^2)=2$ whereas $(4;3,1^3)\cdot (2;1,1,0^2)=0$ contracting the $(-1)$ curve $l-e_1-e_2$ to a point.  Hence the image curve under the morphism induced by $|K_{\w{C}}-\h{E}|$ has a singularity. Therefore the family of linear series arising from 
$\h{L}=|4H+L|$ {\it does not} contribute to a component of $\h{H}_{15,15,5}$. 
\item[(A3)] $\ce\in |3H+4L|$, $p_a(\ce)=15$: In this case $\ce$ is trigonal and by Claim 2,  this case does not occur. 

\item[(A4)] $\ce\in |6H-5L|$, $p_a(\ce)=15$: In this case $\ce\subset\PP^4$ is smooth. For $C_\h{E}\in |6H-5L|$, we have 
 $|K_S+\ce-H|=|3H-4L|.$
We may argue as in the two previous cases (A1) and (A2) to see that the restriction map $$H^0(S,\h{O}(3H-4L))\rightarrow H^0(\ce,\h{O}(3H-4L)\otimes\h{O}_{\ce})$$ is an isomorphism; by $H^0(S,\h{O}(3H-4L-\ce))=H^0(S,\h{O}(-3H+L))=0$ and by Castelnuovo genus bound. 

We also remark that the linear system $|3H-4L|=|2(H-L)+(H-2L)|$ has the fixed part 
$|H-2L|$; $$H^0(S,\h{O}(H-2L))=1 \mathrm{~and~} h^0(S,\h{O}(3H-4L))=h^0(S,\h{O}(2H-2L))=6.$$ Hence the the linear series $|K_{\ce}-\h{E}|$ cut out by the linear system $|3H-4L|$ has 
non-empty base locus $B$, $\deg B= (H-2L)\cdot\ce=(H-2L)\cdot(6H-5L)=1$ and hence $\ce\cdot 2(H-L)=14\neq 15$. This shows that curves in $|6H-5L|$  {\it does not} contribute to a component of $\h{H}_{15,15,5}$.
\vni
\item[(B)] $\deg S=3$ and $S$ is a cone over a twisted cubic in $\PP^3$; $\ce\subset S\subset\PP^4$.

\noindent
Let $\mathbb{F}_3\stackrel{|h+3f|}{\longrightarrow} S\subset\PP^4$ be the minimal desingularization, which contracts $h$ to the vertex $q\in S$
; $h^2=-3, h\cdot f=0, f^2=0$.  Let $\w{C_{{\h{E}}}}\subset \mathbb{F}_3$ be the strict transformation of $C_{{\h{E}}}\subset\PP^4$ and set $\w{C_{{\h{E}}}}\in |ah+bf|$. We have 
$$\w{C_{{\h{E}}}}\cdot (h+3f)=(ah+bf)\cdot (h+3f)=b=13.$$
By \eqref{conevertex1},  $m:=h\cdot\w{\ce}=h\cdot (ah+13f)=-3a+13\ge 0,$ thus $a\le 4$ and hence $a=4$ by Claim 2. Thus $\w{\ce}\in|4h+13f|$, $m=1$, $p_a(\w{\ce})=18$ by adjunction formula, $\ce$ passes through the vertex $q$ of  $S$ and $\ce$ is smooth at $q$.

Set $\h{N}=|4h+13f|$ and let $\Sigma_{15,\h{N}}$ be the Severi variety $\Sigma_{15,\h{N}}$ consisting of curves 
of genus $g=15$ in the linear system $\h{N}$
on $\mathbb{F}_3$. Since a general element of  $\Sigma_{15,\h{N}}$
is a nodal curve, we may assume that $\wce$ has $3$
nodes. Let $\mathbb{F}_{3,3}$ be the blow up of $\mathbb{F}_3$ at three nodes. 
Let $e_i~ (i=1,\cdots, 3)$ be exceptional divisors of the blow up and let
$f_i~ (i=1,\cdots, 3)$ be three typical fibers containing the three nodal points of $\wce$. After resolving the three nodal singularities of $\wce$, we get a smooth curve 
${\widehat{\ce}}\subset\mathbb{F}_{3,3}$. We have
$$\hce\in|4h+13f-\sum2e_i|, \mathrm{~~ and ~set}$$
\begin{align*}
\h{M}:&=|\hce+K_{\mathbb{F}_{3,3}}-(h+3f)|\\&=|(4h+13f-\sum2e_i)+(-2h-5f+\sum e_i)-(h+3f)|\\&=|h+5f-\sum e_i|.
\end{align*}
Since 
$\h{O}_{\mathbb{F}_{3,3}}(\h{M}-\hce)=\h{O}_{\mathbb{F}_{3,3}}(-(3h+8f-\sum e_i))$
and $f\cdot (\h{M}-\w{\ce})<0$,  we have $h^0({\mathbb{F}_{3,3}},\h{O}(\h{M}-\hce))=0$ implying that the restriction map 
$$\rho: H^0({\mathbb{F}_{3,3}},\h{O}(\h{M}))\longrightarrow H^0(\hce,\h{O}(\h{M})\otimes\h{O}_{\hce})$$ is injective. Note that $h^0({\mathbb{F}_{3,3}},\h{O}(\h{M}))=
h^0(\mathbb{F}_3,\h{O}(h+5f))-3=6$.
Set $\h{D}:=\PP(H^0(\hce,\h{O}(\h{M})\otimes\h{O}_{\hce}))$.  Since $X\stackrel{iso}{\cong}\hce\stackrel{bir}{\cong}{\w{\ce}}\stackrel{bir}{\cong}{\ce}$, $\Dd=|K_{\hce}-\Ee|$ is birationally very ample.
If $\rho$ is not surjective, we have
$\dim\Dd\ge 6$ and $\pi(15,6)=13<g=15$, which is a contradiction. Therefore we have 
$$\mathrm{Im}(\rho)= H^0(\hce,\h{O}(\h{M})\otimes\h{O}_{\hce})$$
and the restriction map $\rho$ is surjective. 
Denoting by $\w{f}_i$  the proper transformation of three typical fibers of $f_i$ under the blow up $\FF_{3,3}\rightarrow\FF_3$, we have
$$\tilde{f}_i^2=-1, \tilde{f}_i\cdot e_i=1,  h\cdot\tilde{f}_i=1, \hce\cdot \tilde{f}_i=2, 
(h+5f-\Sigma e_i)\cdot \tilde{f}_i=0.$$
Hence the morphism $\psi$ induced by $\h{M}=|\hce+K_{\mathbb{F}_{3,3}}-(h+3f)|$ contracts  $(-1)$ curves $\tilde{f}_i$ and the image curve $\psi (\hce)\subset\PP^5$ 
acquires singularities. It then follows that $\h{M}\otimes\h{O}_{\hce}=|K_{\hce}(-1)|=\Dd$ is not very ample.
\vni
\item[(C)] $\deg S=4$ and $S$ is smooth. 

\noindent
Note that  $S\cong\PP^2_5$ and 
$\ce$ is smooth; if not, we have
$15=g=\pi_1(13,4)\lneq p_a(\ce)\le \pi(13,4)=18$ 
and  hence $\ce$ is a nearly extremal curve lying on a cubic surface, which has been treated already in the steps (A) and (B).
Setting $\ce\in |(a;b_1,\cdots ,b_5)|$, we have 
$$\deg\ce =3a-\sum b_i=13, \ce^2=a^2-\sum b_i^2=2g-2-K_S\cdot C=41.$$
By Schwartz's inequality,  we have $$(\sum b_i)^2\le 5(\sum b_i^2)$$ and substituting 
$\sum b_i=3a-13$, $\sum b_i^2=a^2-41$ we obtain 
$$5\,({a}^{2}-41)- \left( 13-3\,a \right) ^{2}\ge 0 \Leftrightarrow a=9, 10, 11$$ 
and therefore  we have the following three cases;
\[ (a;b_1,\cdots, b_5)=
\begin{cases}
(9;3^4,2)\\
(10;4^2,3^3)\\
(11;4^5)
\end{cases}
\]
We need to check if  $|K_{\ce}-\h{E}|=|K_S+\ce-H|_{|\ce}=|\ce+2K_S|_{|\ce}$ is very ample.
The restriction map $\rho: H^0(S,\h{O}(\ce+2K_S)\rightarrow H^0(\ce,\h{O}(\ce+2K_S)\otimes\h{O}_{\ce})$
is an isomorphism;  $\ker(\rho)=H^0(S, 2K_S)=0$, $h^0(S,\h{O}(\ce+2K_S))=6$ and by the Castelnuovo genus bound $\pi (15,6)=13<g$ if $\rho$ is not surjective.
Assume the last case among three in the above list; $\ce\in |(11;4^5)|$. We have $$|K_S+\ce-H|=|(11;4^5)-2(3;1^5)|=|(5;2^5)|,$$  whose restriction on $\ce$ does not induce an isomorphism onto its image. To see this, the $(-1)$ curve
$(2;1^5)$ on $S$ is contracted to a point; $(2;1^5)\cdot (5;2^5)=0$ whereas $(11;4^5)\cdot (2;1^5)=2$ and hence the image curve in $\PP^5$ is singular. The verification for the other two cases 
$(9;3^4,2),
(10;4^2,3^3)$ are similar which we omit.  Hence we conclude that $|K_{\ce}-\h{E}|$ is not very ample. 
\item[(D)] 
$\deg S=4$ and $S$ is a cone over an elliptic curve  $E\subset\PP^3$: 

\noindent
Recall that a cone $S\subset\PP^r$ over an elliptic curve $E\subset H\cong\PP^{r-1}$ with vertex outside $H$ is the image of the birational morphism $\mathbb{E}_{r}:=\mathbb{P}(\h{O}_{E}\oplus\h{O}_{E}(r))\rightarrow S\subset\PP^r$ induced by $|\bar{h}|:=|h+rf|$, where $h^2=-r$ and $f$ is the fibre of $\mathbb{E}_{r}\stackrel{\eta}{\rightarrow} E$. 
Let $C\subset S$ be an integral curve of degree $d$ with  the strict transformation  $\w{C}$ under  $\mathbb{E}_{r}\rightarrow S$. 
~Setting $k=\w{C}\cdot f$, we have $\w{C}\equiv kh+df$, $\deg\eta_{|\w{C}}= \w{C}\cdot f=k$ and
\vspace{-8pt}
\begin{equation}\label{conegenus3}
p_a(\w{C})=
(k-1)(d-\frac{kr}{2})+1,~~
0\le \w{C}\cdot h=d-rk=m
\end{equation} 
where $m$ is the multiplicity of $C$ at the vertex.

For $C=\ce$, $d=13$ and $r=4$ we get $(k,m,p_a(\w{\ce}))\in\{(3,1,15), (2,5,10)\}$ by \eqref{conegenus3}. 
In the first case, since  $g=15=p_a(\w{\ce})$, $\w{\ce}\in |3h+13f|=|3\bar{h}+f|$ is smooth and  is a triple covering of an elliptic curve. The second case is not possible since $p_a(\w{\ce})<g$.
We now check if   $$|K_{\mathbb{E}_4}+\w{\ce}-\bar{h}|_{|\w{\ce}}=|(-2\bar{h}+4f+(k\bar{h}+(d-4k)f)-\bar{h}|_{|\w{\ce}}=|5f|_{|\w{\ce}}$$ is very ample. 
\vni
{\bf Claim:} The restriction map 
$$\rho:  H^0(\mathbb{E}_4,|K_{\mathbb{E}_4}+\w{\ce}-\bar{h}|)\to H^0(\w{\ce},K_{\w{\ce}}((-1)))$$
{\it fails} to be surjective: We consider the exact sequence 
$$0\rightarrow\h{I}_{\w{\ce}}(5f)=\h{O}_{\EE_4}(-3\bar{h}+4f)\rightarrow\h{O}_{\EE_4}(5f)\rightarrow\h{O}(5f)_{|\w{\ce}}\rightarrow 0.$$
We have 
$h^0(\EE_4,\h{O}_{\EE_4}(-3\bar{h}+4f))=0$,  $h^0({\EE_4},\h{O}_{\EE_4}(5f))=5$,  $h^1({\EE_4},\h{O}_{\EE_4}(5f))=0$ and by Serre's duality
$$h^1(\EE_4,\h{O}_{\EE_4}(-3\bar{h}+4f))=h^1(\h{O}_{\EE_4}(\bar{h}))=1.$$
Thus $\rho$ 
is not surjective and $h^0(\w{\ce}, \h{O}(5f)_{|\w{\ce}})=6$.

\vni
{\bf Claim:}
$|K_{\w{\ce}}(-1)|=|K_{\mathbb{E}_4}+\w{\ce}-\bar{h}|_{|\w{\ce}}$ is not very ample.
\vni
Let $V:=\mathrm{Im}(\rho)\subset  H^0(\w{\ce},\h{O}(5f)_{|\w{\ce}})$ and we assume $|K_{\w{\ce}}(-1)|$ is very ample inducing an isomorphism $\phi$ onto $X_1$. Since $\PP(V)\subsetneq |K_{\w{\ce}}(-1)|$ and $\dim V=5$ we have the following commutative diagram:
\[
  \begin{tikzcd}
   \mathbb{E}_4\supset\w{\ce}  \arrow{r}{\phi~~ \cong} \arrow[swap]{dr}{\psi} & X_1\subset\PP^5 \arrow{d}{\tau} \\
     & E_1\subset\PP^4
  \end{tikzcd}
\]
(a)  $\psi$ is the morphism on $\EE_4$ induced by the base-point-free $\PP(V)=|5f|$, $\deg E_1=\bar{h}\cdot 5f=5$, $\deg\psi_{|\w\ce}=\w{\ce}\cdot f=(3\bar{h}+f)\cdot f =3$. $E_1$ is elliptic by Claim 2.

\noindent
(b)
$\tau$ is the projection map with center of projection 
$p_1\in\PP^5$ corresponding to $\PP(V)$, i.e. the intersection of all hyperplanes corresponding to divisors in $\PP(V)$,  inducing a morphism $X_1\stackrel{\tau}{\to}E_1$, $\deg\psi=\deg\tau =3$, $\tau\circ\phi=\psi$ and $p_1\notin X_1$ since $\PP(V)$ is base-point-free.

Let $T_1\subset\PP^5$
be a  cone over $E_1$ with vertex $p_1$. 
$T_1$ is the image of the morphism on $\mathbb{E}_5:=\mathbb{P}(\h{O}_{E_1}\oplus\h{O}_{E_1}(5))$ induced by  $|\w{h}|:=|h+5f|$ and we have $X_1\subset T_1$.
Let $\w{X}_1$ be the strict transformation of $X_1$ via $\EE_5\rightarrow T_1$. 
Setting $k=\w{X}_1\cdot f$ we have $\w{X}_1\equiv k\w{h}+(d-5k)f$.  By \eqref{conegenus3}
(for $d=15$, $r=5$) we get
$(k,m,p_a({\w{X}_1}))\in\{(3,0,16), (2,5,11)\}$. 
In the first case $p_a(\widetilde{X})=p_a(X)=16>g=15$, a contradiction since $X_1$ and $\w{X}_1$ are smooth.  The second 
case is not possible since $p_a(\w{X})<g$. Thus we deduce that $|K_{\w{\ce}}(-1)|$ is not very ample.
\item[(E)] $\deg S=4$ and $S$ is singular with isolated singularities; $\ce\subset S\subset\PP^4$.
We see that $\ce$ is smooth; otherwise we have
$15=g=\pi_1(13,4)\lneq p_a(\ce)\le \pi(13,4)=18$,
hence $\ce$ is a nearly extremal curve lying on a cubic surface which has been treated already in the steps (A) and (B).
We assume that there is a smooth curve $C$ -- which we may take as $\ce\subset\PP^4$ under our current situation --  of degree $d=13$ and genus $g=15$ on a singular del-Pezzo surface $S$ with isolated singularities. We also assume that the  dual curve of $\ce$  is a smooth curve  $X\subset\PP^5$.  By Proposition \ref{b1}, we let  $\{C_t\}$ be a one parameter flat family of (smooth) curves with $C_0=C=\ce$ lying on a singular del Pezzo $S\subset\PP^4$ and $\{C_t; t\neq 0 \},$ lying on 
smooth del Pezzo surfaces. If the dual curve $X=C_0^\vee\subset\PP^5$ is smooth, dual curves $\{C_t^\vee; t\neq 0\}$ are also smooth since singular curves cannot specialize to a smooth curve $X=C_0^\vee$. However
this is contradictory to what we have verified in (C), i.e. every smooth curve $C_t\subset \PP^4$ with $(d,g)=(13,15)$ lying on a smooth del Pezzo has its dual curve in $\PP^5$ which is always singular. 
\vni
{\bf Conclusion:} We have exhausted all the possibilities for the surfaces $S\subset \PP^4$ on which dual curves $\ce$ of 
$X\in\h{H}_{15,15,5}$ may sit. Our lengthy discussion in parts (A)--(E) shows that the only case such that the residual series of the hyperplane series of $\ce\subset\PP^4$ is very ample - among all the possibilities for the surface $S$ - is the case  (A1);  $\ce\subset S\subset\PP^4$ lies a smooth cubic surface $S$, $p_a(\ce)=18$, $\ce\in \Sigma_{15,\h{M}}$ where $\h{M}=|5H-2L|$. Part (A1) also shows that the curve $X$ corresponding to a general  element in the Severi variety $\Sigma_{15,\h{M}}$ lies on a smooth del Pezzo surface in $\PP^5$; cf. \eqref{del}.  From  \eqref{octic}, we see that a general $X\in\HO{15,15,5}$ has a plane model of degree $8$ with an ordinary triple point and lines through the triple point  cut out a base-point-free  $g^1_5$. $X$ does not  have $g^1_4$ by Castelnuvo-Severi inequality. 
\end{proof}
\subsection{Moduli map $\mu: \h{H}_{15,15,5}\to \Mm_{15}$ }
In this subsection, we show that two smooth curves in $\HO{15,15,5}$ are isomorphic as abstract curves if and only if 
they are projectively equivalent. In order to prove this seemingly plausible assertion, we need several preparatory results which occupy a major part of this subsection.

Let  $X\subset \PP^r$, $r\ge 2$, be a smooth curve of genus $g\ge 2$. Since $X$ has only finitely many automorphisms, the set $G:=\{h\in \Aut(\PP^r)| h(X)=X\}$ is a finite group. Hence the set $$\Aut(\PP^r)X:=\{Y\subset\PP^r | Y=\sigma (X) \mathrm{~ for ~some~ }\sigma\in \Aut (\PP^r)\}$$ consisting of all curves $Y\subset \PP^r$ projectively equivalent to $X$ is an irreducible  quasi-projective variety isomorphic to $\Aut(\PP^r)/G$. 
\begin{theorem}\label{m15}
Let $\mu: \h{H}_{15,15,5}\to \Mm_{15}$ denote the moduli map. Then we have $\mu^{-1}(\mu(X))  =\mathrm{Aut}(\PP^5)X$ for a general $X\in \h{H}_{15,15,5}$. \end{theorem}
We review a few basic facts about residual schemes which we use in Theorem \ref{t1}, an essential step toward the proof of Theorem \ref{m15}.
\begin{remark}\label{residue}
Let $S$ be a projective variety, $D$ an effective Cartier divisor of $S$ and $Z\subset S$ a zero-dimensional scheme. The residual scheme $\Res_D(Z)$ of $Z$ is by definition the closed subscheme of $S$ with $\Ii_Z:\Ii_D$ as its ideal sheaf.
We have 
$\Res_D(Z)\subseteq Z$ and  
\begin{equation}\label{ele}
\deg (Z) =\deg(Z\cap D)+\deg(\Res_D(Z)).
\end{equation} 
We also have;
\begin{itemize}
\item
$\Res_D(Z)=Z$ if and only if $Z\cap D=\emptyset$.
\item
$\Res_D(Z)=\emptyset$ if and only if $Z\subset D$. 
\item
If $Z$ is a finite set then $\Res_D(Z) =Z\setminus Z\cap D$. 
\item
If $Z =Z_1\cup Z_2$ with $Z_1\cap Z_2=\emptyset$ then $\Res_D(Z) =\Res_D(Z_1)\cup \Res_D(Z_2)$. 
Hence to compute
$\Res_D(Z)$ we may look separately at the connected components of $Z$. 
\item Let   $D_1$ and $D_2$ be effective Cartier divisors of $S$. Call $D_1+D_2$ the sum of effective divisors. We have
\begin{enumerate}
\item[$\diamond$]
$D_1+D_2=D_1\cup D_2$ if and only if $D_1$ and $D_2$ have no common irreducible component. 
\item[$\diamond$] $\Res_{D_1+D_2}(Z) =\Res_{D_1}(\Res_{D_2}(Z))$. 
\item[$\diamond$] If $Z\subset D_1+D_2$, then $\Res_{D_1}(Z)\subset D_2$ and $\Res_{D_2}(Z)\subset D_1$.
\end{enumerate}
\item
Take a smooth point $p$ of $S$ and call $2p$ (resp. $3p$) the closed subscheme of $S$ with $(\Ii_p)^2$ (resp. $(\Ii_p)^3)$ as its ideal sheaf. 
We have 
\begin{enumerate}
\item[$\diamond$]
$(2p)_{\mathrm{red}} =(3p)_{\mathrm{red}} =\{p\}$, $\deg (2p) =1+\dim S$,$\deg (3p) =\binom{\dim S +2}{2}$. 
\item[$\diamond$] For a smooth point $p$ of $D$, $\Res_D(2p) = p$ and $\Res_D(3p) =2p$. 
\item[$\diamond$]
For a singular point $p$ of $D$, we have $2p\subset D$ and $\Res_D(2p) =\emptyset$.
\end{enumerate} 
\item Let $\Rr$ be a line bundle on $S$. We have the following exact sequence, usually called the residual exact sequence of $Z$ with respect to $D$:
\begin{equation}\label{eqt0}
0 \to \Ii_{\Res_D(Z)} \otimes \Rr(-D) \to \Ii_Z\otimes \Rr \to \Ii_{Z\cap D,D}\otimes \Rr_{|D}\to 0.
\end{equation}
\end{itemize}
For further details on  residual schemes, readers are advised to consult \cite[Section 2]{adgeo} and references therein.
\end{remark}
\begin{remark}\label{aa1}
(a) Let $g^r_d$, $r\ge 2$, be a base-point-free linear series on a smooth curve $X$ which is not compounded. Then its monodromy is the full symmetric group $S_d$
(\cite[p. 111]{ACGH} or \cite[pp. 85-86]{he})). Later in this section, this will be used in the following way.
Let $L_i$ ($1\le i\le s$) be linear series' on $X$, possibly incomplete.
Since the monodromy group of the $g^r_d$ is the full symmetric
group, there is a non-empty open subset $U_i$ of $g^r_d$ such that for each $V\in U_i$ all subsets of $V$ with the same cardinality impose the same number of conditions on
$L_i$.
Every element of $U_1\cap \cdots \cap U_s$ has the same property for all $L_i$, $1\le i\le s$.
\vskip 1.5pt
\noindent
(b) In the next theorem we use a well-known and strong tool known as Horace method. Let $Y\subset \PP^2$ be an integral curve  of degree $d$ and $v: X\to Y$ the normalization map. Let $Z\subset \PP^2$ be  the zero-dimensional scheme associated to $v$, i.e. the scheme such that $H^0(\PP^2,\Ii_Z(d-3)) \cong H^0(X,K_X)$. Take a base point free  linear system $\Tt$ on $X$ and take a general $E\in \Tt$. Set $m:=\deg(E)$ and $B:=v(E)$. Since $\Tt$ is base point free $E$ is formed by $m$ distinct points, $\#B=m$ and $B\cap Z=\emptyset$. Set $W:= B\cup Z$. Knowing the integer $h^1(\PP^2,\Ii_W(d-3))$ provides a key information on $\Tt$. For instance, if we take $\Ll:= v^\ast (\Oo_Y(1))$ and $\Tt =|\Ll|$ and take as $B$ the union of $d$ collinear smooth points of $Y$
we get that $h^0(X,\Ll)=3$ if and only if $h^1(\PP^2,\Ii_W(d-3)) =h^1(\PP^2,\Ii_Z(d-4)) =0$. To say something about the Brill-Noether theory on $X$,  e.g. in the next theorem, given a  plane curve $Y$ with $\deg Y=d=8$ having certain prescribed singularity types (one ordinary triple point and three nodes or cusps), we want to show that  there is a {unique }$g^1_5$ evincing the gonality of $X$, exactly three base-point-free $g^1_6$'s and that $Y$ {has} a unique degree $d$ plane model. For this, we need to give upper bounds on $h^1(\PP^2,\Ii_{B\cup Z}(d-3))$ {only from the information $\#B=m$.}  Specifically, to show  that $X$ has a unique $g^1_5$ and exactly three $g^1_6$ we take $B$ with $\#B \in \{5,6\}$ and assume $h^1(\PP^2,\Ii_{B\cup Z}(5)) >0$. By \cite[Lemma 34]{bgi} there is a line $L$ such that $\deg(L\cap (Z\cup B)) \ge 7$ and then use a residual exact sequence and the explicit form of scheme $Z$ to conclude the proof of the characterization (or description) of the base-point-free $g^1_t$'s, $t\le 6$, on $X$.  This part is a key step to prove that $X$ has a unique $g^2_8$, an essential step toward the proof of the description of the general fiber of the moduli map in genus $15$; Theorem \ref{m15}. This approach, the study of the cohomology group of a certain zero-dimensional scheme $W\subset \PP^2$ using low degree curves, say a line $L$, with $\deg(L\cap W)$ very high is usually called the ``Horace Method"; cf.  \cite{hirscho}. 
\end{remark}
\begin{theorem}\label{t1}
Fix a set $A\subset \PP^2$ such that $\#A=3$ and $A$ is not collinear. Fix $p\in \PP^2\setminus A$ such that $p$ is not contained in a line spanned by $2$ points in $A$.
Let $Y\subset \PP^2$ be an integral degree $8$ curve whose only singularities are either an ordinary node or an ordinary cusp at each point of $A$ with an ordinary triple point at $p$.
Let $v: X\to Y$ denote the normalization map. Then $X$ has genus $15$. Moreover
\begin{itemize}
\item[(a)] $X$ is $5$-gonal and the only $g^1_5$ on $X$ is induced by the pencil of lines through the ordinary triple point $p$.
\item[(b)] $X$ has exactly $3$ base-point-free $g^1_6$, which are induced by the pencils of lines through one of the points of $A$.
\item[(c)] Set $\Ll:= v^\ast (\Oo_Y(1))$. Then $h^0(X,\Ll)=3$ and $|\Ll|$ is the unique $g^2_8$ on $X$.
\end{itemize}
\end{theorem}
\begin{proof}
$X$ has genus $g=15$ by the assumption that $Y$ has $3$ ordinary nodes or ordinary cusps at points in $A$ and an ordinary triple point at $p$.
Let $2p$ the closed subscheme of $\PP^2$ with $(\Ii_p)^2$ as its ideal sheaf. We have $\deg (2p)=3$ and $(2p)_{\mathrm{red}} =\{p\}$. Set and fix $Z:=A \cup 2p$ once and for all; $\deg (Z)=6$. We note that $Z$ is the conductor of the normalization map,
i.e. the complete linear system $|K_X|$ is induced by $|\Ii_Z(5)|$. Thus to prove that $h^0(X,\Ll)=3$ it is sufficient to show that $Z$ imposes $6$ independent conditions on $|\Oo_{\PP^2}(4)|$, i.e.  $h^1(\PP^2,\Ii_Z(4)) =0$. 
Recall that for a  degree $6$ zero-dimensional scheme $F\subset \PP^2$,  $h^1(\PP^2,\Ii_F(4))=0$ if and only if
$F$ is not contained in a line; cf. \cite[Lemma 34]{bgi}. 
Therefore we have $h^1(\PP^2,\Ii_Z(4)) =0$ by the assumption on $Z$ and hence $h^0(X, \Ll)=3$. We can also deduce easily that the line bundle $\Rr$ described in (a) 
and the $3$ line bundles $\Ll_1,\Ll_2,\Ll_3$ described in (b) are complete pencils.
After the characterization of the (unique) $g^1_5$ and the (only) three $g^1_6$'s in the following part (a) and (b), the uniqueness of the $g^2_8$ is shown in (c). 

\vni
For a line $L\subset \PP^2$, $\deg (L\cap Z)\le 3$ by our assumptions on $A\cup \{p\}$. Note that
\begin{enumerate}
\item[(i)]
 $\deg (L\cap Z)=3$  if and only if $L$ is one of the $3$ lines, say $R_1,R_2,R_3$, containing $p$ and one of the points of $A$. 
 \item[(ii)] $\deg (L\cap Z)=2$ if and only if either $L$ is one of the lines, $L_1,L_2,L_3$, containing $2$ of the points of $A$ or $p\in L$ and $L\notin \{R_1,R_2,R_3\}$. 
 \end{enumerate}
 Recall that a conic $D\subset \PP^2$ contains $2p$ if and only $D$ is singular at $p$ if and only if $D$ is a union of two lines intersecting at $p$ or a double line through $p$.
Thus 
\begin{equation}\label{h1z}
\begin{cases}
h^0(\PP^2,\Ii_Z(2)) =0,\\
h^1(\PP^2,\Ii_Z(2)) =0 \mathrm{ ~since ~} \deg (Z)=6 \mathrm{ ~and 
~therefore }\\
h^1(\PP^2,\Ii_Z(t)) =0 \textrm{ for all } t\ge 3.
\end{cases}
\end{equation}
\noindent
(a) $X$ has no pencil of degree four or less by  the Castelnuovo-Severi inequality \cite[Theorem 3.5]{Accola1} and  hence $X$ is $5$-gonal with a $g^1_5$ cut out by lines through $p$.  Take a complete base-point-free pencil $\Rr$ on $X$ such that $\deg (\Rr) = 5$.  Fix a general $E\in \Rr$ and set $B:= v(E)$. We note that 
\begin{enumerate}
 \item[(ai)]
$B\cap Z=\emptyset$ since $\Rr$ is base-point-free and $E$ is general, 
 \item[(aii)]
$h^1(\PP^2,\Ii_{Z\cup B}(5)) >0$ since $|K_X|$ is induced by $|\Ii_Z(5)|$ and $\Rr$ is a pencil, 
 \item[(aiii)]$h^1(\PP^2,\Ii_{Z\cup B'}(5)) =0$ for all $B'\subsetneq B$ since $\Rr$ is complete and base-point-free, 
  \item[(aiv)]$|\Ii_B(2)|\ne \emptyset$ since $\#B=5$. 
\end{enumerate}
Fix a general $C\in |\Ii_B(2)|$. 
Set $M:= C\cap (Z\cup B)$.
Since $B\subset C$, we ha{ve $\Res_C(Z\cup B) =\Res_C(Z)$. We consider the residual exact sequence
\begin{equation}\label{eqt1}
0 \to \Ii_{\Res_C(Z)}(3)\to \Ii_{Z\cup B}(5) \to \Ii_{M,C}(5) \to 0
\end{equation} 
of $Z\cup B$ with respect to $C$. Since $\Res_C(Z)\subseteq Z$ and  $h^1(\PP^2,\Ii_Z(2)) =0$ by \eqref{h1z}, we have $h^1(\PP^2,\Ii_{\Res_C(Z)}(3)) =0$. Hence by (aii), the long cohomology sequence of \eqref{eqt1}  yields $$h^1(C,\Ii_{M,C}(5)) >0.$$ 

We first assume that $C$ is smooth. Since $\deg M\le\deg (Z\cup B) =11$,
$\deg (\Oo_C(5)) =10$ and $C\cong \PP^1$, we have $h^1(C,\Ii_{M,C}(5))=0$, a contradiction. 

We next assume that $C$ is singular (a priori even a double line), say $C=C_1+C_2$ with  $\deg (C_1\cap (Z\cup B)) \ge \deg (C_2\cap (Z\cup B))$ if $C_1\ne C_2$. 

Since $M\subset C_1+C_2$, we have $\Res_{C_1+C_2}(M) =\emptyset$, 
$\Res_{C_2}(\Res_{C_1}(M)) =\emptyset$ and hence $\Res_{C_1}(M)\subset C_2$.  By the basic property
\eqref{ele} on residual schemes in Remark \ref{residue}, $\deg (\mathrm{Res}_{C_1}(M))
=\deg (M) -\deg (M\cap C_1)$.  Since $\deg (M)\le 11$ and $\deg (M\cap C_1)\ge \deg(M\cap C_2)$ by assumption, we have $\deg (M\cap C_2)\le 5$. Since $\Res_{C_1}(M)\subset C_2$ and $\Res_{C_1}(M)\subseteq M$,
we have $$\deg (\Res_{C_1}(M)) =\deg (\Res_{C_1}(M)\cap C_2) \le \deg (M\cap C_2) \le 5.$$
 Since $h^1(C,\Ii_{M,C}(5)) >0$, we have $h^1(\PP^2,\Ii_M(5))>0$ by usual cohomology computation of the  sequence $0\rightarrow \Ii_C(5)\rightarrow \Ii_M(5) \rightarrow\Ii_{M,C}(5)\rightarrow 0$. Consider the residual exact sequence of $M$ with respect to the line $C_1$;
\begin{equation}\label{eqt2}
0 \to \Ii_{\mathrm{Res}_{C_1}(M)}(4)\to \Ii_M(5) \to \Ii_{C_1\cap M,C_1}(5)\to 0.
\end{equation}
Since $\deg (\Res_{C_1}(M)) \le 5$, we have $h^1(\PP^2,\Ii_{\Res_{C_1}(M)}(4)) =0$ \cite[Lemma 34]{bgi}. Thus the long cohomology exact sequence of \eqref{eqt2} and $h^1(\PP^2,\Ii_M(5))>0$ yield $$h^1(C_1,\Ii_{C_1\cap M}(5)) >0 ~ \textrm{implying} ~\deg (C_1\cap M)\ge 7.$$  From this we have the following two cases; 
\[
\begin{cases} B\subset C_1 \mathrm{ ~and ~}\deg (Z\cap C_1)\ge 2 \mathrm{~ or ~} \\
\#(B\cap C_1)=4 \mathrm{ ~and~ } \deg (Z\cap C_1)=3.
\end{cases} 
\]
The latter possibility is excluded for a general $E\in |\Rr|$, because only $3$ lines, $R_1,R_2,R_3$
intersect $Z$ in a degree $3$ schemes. Thus $B\subset C_1$ and $\deg (Z\cap C_1)=2$. Since there are only $3$ lines, $L_1,L_2,L_3$, intersecting $Z$ in  a degree $2$ scheme and not containing $p$, we get $p\in C_1$, concluding the proof of (a). 
\vni
(b)
Fix a base-point-free line bundle $\Mm$ on $X$ such that $\deg (\Mm)=6$. Since $X$ has a unique $g^1_5$ by part (a), we have $h^0(X,\Mm)=2$. To see this, $X$ has no birationally very ample $g^2_6$
by the uniqueness of $g^1_5$ or by genus reason. $X$ does carry a compounded $g^2_6$ either since $C$ is neither trigonal nor bi-elliptic. 

Fix a general $G\in |\Mm|$ and set $F:= v(G)$. Since $\Mm=g^1_6$ is base-point-free and complete, we have $h^1(\PP^2,\Ii_{Z\cup F}(5))>0$
and $h^1(\PP^2,\Ii_{Z\cup F'}(5)) =0$ for all $F'\subsetneq F$. Choose any $F'\subset F$ formed by $5$ points and  take $D\in |\Ii_{F'}(2)|\neq\emptyset$. Set $N:= D\cap (Z\cup F)$. Assume $D$ is smooth thus no $3$ among $F'$  are collinear. 
Consider the residual exact sequence of  $Z\cup F$ with respect to $D$;
\begin{equation}\label{eqt3}
0\to \Ii_{\mathrm{Res}_D(Z\cup F)}(3)\to \Ii_{Z\cup F}(5) \to \Ii_{N,D}(5)\to 0.
\end{equation}
Since $\#(F\setminus F\cap D)\le 1$, $\deg{\mathrm{Res}_D(Z\cup F)}\le 7$. Since no  $5$ among $\mathrm{Res}_D(Z\cup F)$ is collinear, $h^1(\PP^2,\Ii_{\mathrm{Res}_D(Z\cup F)}(3))=0$ by \cite[Lemma 34]{bgi}.}
 Thus the long cohomology sequence from  \eqref{eqt3} gives $h^1(D,\Ii_{N,D}(5))>0$.  
Since $D\cong \PP^1$, $\deg(\Oo_D(5)) =10$ and $h^1(D,\Ii_{N,D}(5)) >0$, we get $$\deg(N)\ge 12.$$
On the other hand, we have $\deg (D\cap Z)\le 5$. To see this, we assume 
$\deg (D\cap Z)\ge 6$ and hence $D$ is smooth. From the exact sequence $0\to\Ii_{D\cap Z, D}(2)\to\Oo_D(2)\to\Oo_{D\cap Z}(2)\to 0$ we have $h^1(D,\Ii_{D\cap Z,D}(2)) >0$ following from $h^0(D,\Oo_D(2))=5$ and $h^0(\Oo_{D\cap Z})=\deg (D\cap Z)\ge 6$. From the 
exact sequence $0\rightarrow\Ii_{D}(2)\rightarrow \Ii_{D\cap Z}(2)\rightarrow\Ii_{D\cap Z,D}(2)\rightarrow 0$
and by $h^i(\PP^2,\Ii_D(2))= h^i(\PP^2,\Oo_{\PP^2}) =0~ (i\ge 1)$, we get $h^1(\PP^2,\Ii_{D\cap Z}(2)) >0$.
Since $D\cap Z\subseteq Z$, we get $h^1(\PP^2,\Ii_Z(2)) >0$
contrary to \eqref{h1z} concluding $\deg (D\cap Z)\le 5$, thus $$\deg (N)\le 11.$$ 

This contradiction shows that $D$ is not smooth. Set $D=D_1+D_2$. Exactly as  in step (a) we may prove the existence of a line $D_1$ with $\deg (D_1\cap N)\ge 7$. Since $\Mm$ is not induced by the pencil of lines through $p$, $p\notin D_1$. For a general $G\in|\Mm|$ we have $D_1\notin \{L_1,L_2,L_3\}$. Thus $F\subset D_1$ and $D_1$ contains one of the points of $A$, concluding the proof of  (b).

\vni
(c) We only need to prove the uniqueness part. Take a  line bundle $\Nn$ on $X$ such that $h^0(\Nn)\ge 3$ and $\deg (\Nn)\le 8$. Since $X$ has only finitely many base-point-free 
$g^1_6$'s by part (b), $\deg(\Nn)=8, h^0(X,\Nn)=3$ and $\Nn$ is base-point-free.
Part (a)  implies that $|\Nn|$ is not compounded; if it were then either $X$ is $4$-gonal or a double covering of a smooth plane curve of degree $3$, which may be excluded by the Castelnuovo-Severi inequality. We want to apply Remark \ref{aa1}(a) to the linear series $|\Nn|$. 

Fix a general $V\in |\Nn|$ and set $U:= {v}(V)$. To conclude the proof we need to prove that $U$ is formed by $8$ collinear points. For a general $V$ we have
$Z\cap {v}(V) =\emptyset$. Since $h^0(X,\Nn)=3$, we have
\begin{equation}\label{net}
h^1(\PP^2,\Ii _{Z\cup U}(5))\ge 2.
\end{equation}

\vni
Observation 1:
\begin{enumerate}
\item[(0)] Since $V$ is general, Remark \ref{aa1}(a) implies that if $\#(U\cap L)\ge 3$ for some line $L$, then $U\subset L$, concluding the proof. 
\item[(i)] Thus from now on we may assume that {\it no 3 points among $U$ are collinear}. 
\item[(ii)] Suppose $\#(U\cap D)\ge 6$ for some conic $D$, i.e. $U\cap D$ fails to impose independent conditions on $|\Oo_{\PP^2}(2)|$. 
Hence by Remark \ref{aa1}(a) we have $U\subset D$.
\end{enumerate}
\vni
Observation 2:
\begin{enumerate}
\item[(0)]
The zero-dimensional scheme $Z =2p\cup A$ has degree $6$ and is not contained in a conic;  just because $A$ is not formed by $3$ collinear points and any conic containing $Z$ is singular at $p$.
\item[(i)] We set 
$\Zz:=\{D\in |\Oo_{\PP^2}(2)| | \deg Z\cap D=5\}, \Zz_1:=\{D\in \Zz | 2p\subset D\},$  $\Zz_2:= \Zz\setminus \Zz_1.$
Each $D\in \Zz_1$ is singular at $p$ and hence $\#\Zz_1=3$;  each $D\in \Zz_1$ is the union of two lines through $p$ containing one of the points of $A$.
For $D\in \Zz_2$, 
$D\cap Z=A\cup w$, where $w$ is degree $2$ connected zero dimensional subscheme with $p$ as its reduction.
Since no conic contains $Z$, $A\cup w$ uniquely determines $D$. Thus $\Zz_2$ is a $1$-dimensional family and $\Zz$ is an algebraic family of dimension $1$.
\end{enumerate}
\vni

Given $U={v}(V)$, let $D_U\in |\Oo_{\PP^2}(2)|$ be such that $\#(D_U\cap U)$ is maximal. Note that $\#(U\cap D_U)\ge 5$ since $\dim |\Oo_{\PP^2}(2)| =5$. By Observation 1(i), $D_U$ is a smooth conic. Bezout gives $\#(D_U\cap v(X)) \le 16$.  Thus $D_U$ contains at most one other set $U'$ with $\#U' =8$ with $U'=\#v(V') =8$ for some $V'\in |\Nn|$ and $U'\cap U=\emptyset$.
\[
\begin{cases}
 \#(D_U\cap U) =5 \textrm{ or }
 \\ U\subset D_U.
\end{cases}
\]
Now we know that there are two dimensional family of conics $$\Uu:=\{D_U|
U={v}(V), V\in\Nn, \#(D_U\cap U)\ge 5\}\subset|\Oo_{\PP^2}(2)|.$$
On the other hand, the family of conics $$\Zz=\{D|\deg(D\cap Z)=5\}\subset|\Oo_{\PP^2}(2)|$$
moves only in one dimensional family by  Observation 2(i). Hence  for general $D_U\in\Uu$, we have 
\begin{equation}\label{le4}\deg (Z\cap D_U)\le 4.
\end{equation} 
(c1) Assume $\#(D_U\cap U)=5$ for general $D_U\in\Uu$. Recall that by  Observation 1(i), $D_U$ is a smooth conic and  we set $W:= U\setminus D_U\cap U$  consisting of $3$ non-collinear points. Note that $$\deg (D_U\cap (Z\cup U))=\deg (D_U\cap Z)+\deg (D_U\cap U) \le 9$$ and hence $\deg\Ii_{D_U\cap (Z\cup U),D_U}(5)\ge 1$ implyig
 $h^1(D_U,\Ii_{D_U\cap (Z\cup U),D_U}(5)) =0$.
  Consider the residual exact sequence of $Z\cup U$ with respect to $D$:
\begin{equation}\label{eqt4}
0\to \Ii_{\mathrm{Res}_D(Z\cup U)}(3)\to \Ii_{Z\cup U}(5) \to \Ii_{D\cap (Z\cup U),D}(5)\to 0.
\end{equation}
From the long cohomology exact sequence of \eqref{eqt4} and \eqref{net} we have $$h^1(\PP^2,\Ii_{\mathrm{Res}_D(Z\cup U)}(3))\ge 2$$ and hence 
\begin{equation}\label{zw}
h^1(\PP^2,\Ii_{Z\cup W}(3))\ge h^1(\PP^2,\Ii_{\mathrm{Res}_D(Z\cup U)}(3))\ge 2.
\end{equation}
 Take a line $L$ containing $2$ points of $W$ and set $\{o\}:= W\setminus W\cap L$.  We consider the residual exact sequence of $Z\cup W$ with respect to $L$; note that $\Res_L(Z\cup W)=\Res_L(Z)\cup\{o\}$. Hence we have the exact sequence
\begin{equation}\label{eqanew2}
0 \to \Ii_{\Res_L(Z)\cup \{o\}}(2)\to \Ii_{Z\cup W}(3) \to \Ii_{(Z\cup W)\cap L,L}(3) \to 0.
\end{equation}
Since $\Res_L(Z)\subseteq Z$ and $h^1(\PP^2, \Ii _Z(2)) =0$ (by \eqref{h1z}), we have 
$$h^1(\Ii_{\Res_L(Z)\cup \{o\}}(2))\le 1.$$ From the long cohomology exact sequence of \eqref{eqanew2}  together with \eqref{zw}, we have
$h^1(L,\Ii_{(Z\cup W)\cap L,L}(3))\ge 1$. Thus $\deg ((Z\cup W)\cap L)) \ge 5$. Since $\deg (W\cap L) =2$, we obtain $\deg (Z\cap L)\ge 3$. Thus $L$ is one of the $3$ lines spanned by $p$ and one of the points of $A$, say $L=R_1$; remember that the three lines $R_i, i=1,2,3$
does not depend on the choice of  $V\in|\Nn|$ and $U=v(V)$. 
On the other hand, since $R_1\cap v(X)$ is a finite set, for a general $V\in |\Nn|$ we have $v(V)\cap R_1=U\cap R_1=\emptyset$. However 
we took the line $L$ containing two of the points of $W\subset U$, a contradiction.
\vni
(c2) Assume $U\subset D_U$. 
In this case we have $\mathrm{Res}_{D_U}(Z\cup U)\subseteq Z$ and hence $h^1(\PP^2,\Ii_{\mathrm{Res}_{D_U}(Z\cup U)}(3)) =0$ since $h^1(\PP^2, \Ii_Z(3))=0$ by \eqref{h1z}. The long cohomology exact sequence of \eqref{eqt4} together with \eqref{net}, i.e. $h^1(\PP^2,\Ii _{Z\cup U}(5))\ge 2$ 
yields 
$$h^1({D_U},\Ii_{{D_U}\cap (Z\cup U),{D_U}}(5)) \ge 2.$$  Recalling  $\deg (Z\cap {D_U})\le 4$,
on a smooth conic ${D_U}$ we have $\deg ({D_U}\cap (Z\cup U)) \le 12$, $\deg\Ii_{{D_U}\cap (Z\cup U),{D_U}}(5)\ge -2$ and hence
$h^1({D_U},\Ii_{{D_U}\cap (Z\cup U), {D_U}}(5))\le 1,$
a contradiction.
\vni
{\bf Conclusion}: (c1) and (c2) show that there is a subset $B\subset U$ with 
$\#B=3$ and $B$ is collinear, i.e. $v^{-1}(B)$ fails to impose independent conditions on $\Ll=v^*(\Oo_{\PP^2}(1))$, hence any subset $B'\subset U$ with $\#B'=3$ is collinear by Remark \ref{aa1} and we are done. 
\end{proof}
\noindent
{\it{Proof} of Theorem \ref{m15}}: Recall that  a general element of $X\in\HO{15,15,5}$ has a plane model of 
degree eight with one ordinary triple point and three nodes as its only singularities; cf. part  (A1) in the proof of Theoerm \ref{main15}. We also recall that 
a curve with a plane model $C$ of degree $8$ with such prescribed singularities is embedded into $\PP^5$ as a smooth curve of degree $15$ and genus $g=15$ in the following way: 
\vni
(i) Blowing up $\PP^2$ at four (singular)  points in general position in $\PP^2$ and  then take the strict transformation $\w{C}$ of $C$ in $\PP^2_4$ under this blow up.
\vni
(ii) We then embed $\PP^2_4$ and $\w{C}$ by the anti canonical  system $|-K_{\PP^2_4}|=|(3;1^4)|$ to get a smooth del Pezzo $S\subset\PP^5$ and smooth $\w{C}\cong X\subset\PP^5$. 

Now we fix four points $A\subset\PP^2$ in general linear position. Let $C_i,~ 1\le i\le2$ be two plane curves of degree $8$ with 
one triple point and three nodes with $\mathrm{Sing}(C_i)=A$.
By Theorem \ref{t1}, we have the following equivalent conditions:
\begin{itemize}
\item[(a)] Two curves $X_i\subset\PP^5, i=1,2$ such that $X_i\cong\w{C}_i$ are isomorphic.
\item[(b)] Two singular plane models $C_i$ of $X_i$ are projectively equivalent under a projective motion of $\PP^2$ inducing a permutation on  the set $A\subset\PP^2$. 
\item[(c)] $X_i$ lies on a same smooth del Pezzo surface $S\subset\PP^5$.
\item[(d)] There exists  $\tau\in\Aut(S)$ such that $X_1=\tau(X_2)$. 
\end{itemize}
Note that for a smooth del Pezzo $S\subset\PP^5$ and  $\tau\in\Aut(S)$, there is $\beta\in \Aut(\PP^5)$ such that $\beta_{|S} =\tau$ since $S$ is anticanonically embedded in
$\PP^5$. Hence we have $\beta(X_2)=\tau(X_2)=X_1$.
\qed
\vspace{-12pt}
\section{Curves of genus $g=16$}
\subsection{Reducibility of $\HO{15,16,5}$}
The aim of this subsection is to prove the following reducibility result for  $\HO{15,16,5}$. 
                                                                                                                                                                                                                                                                                                                                                                                                                                                                                                                                                                                                                                                          
\begin{theorem}\label{main16}
$\HO{15,16,5}$ has $3$ irreducible components, $\Gamma _1$, $\Gamma _2$ and $\Gamma_3$, described as follows:

\begin{enumerate}
\item  $\dim \Gamma _1=68$, every $X\in \Gamma _1$ lies on a smooth quartic surface and is trigonal.
\item $\dim \Gamma_2=64$, every $X\in \Gamma _2$ lies on a smooth quartic surface and is pentagonal.
\item $\dim \Gamma_3=65$, every $X\in \Gamma_3$ is ACM, lies on a quintic surface,  is $6$-gonal and  $K_X\cong \Oo_X(2)$.\end{enumerate}
\end{theorem}
\vni

Take $X\in \HO{15,16,5}$. 
By the Castelnuovo's genus bound, $\pi(15,6)=13$ and $\pi(15,5)=18$, hence  $X$ is linearly normal. Since $\pi_1(15,5)=16=g<\pi (15,5)$, $X\subset S\subset\PP^5$ where $S$ an irreducible surface with $4\le \deg S\le 5$ by \cite[Theorem 3.15]{he}. We start making observations for the case $\deg S=4$.
\vni
(A) $\deg S=4$ case:

We may assume that $S$ is a smooth rational normal scroll by Remark \ref{cone1} (c). 
Let $X\in |aH+bL|$ on $S$. By solving \eqref{sd} - the degree and genus formula  for $d=15$ and $p_a(X)=16$ - we get $(a,b)\in \{(3,3), (5,-5)\}$. 
Thus we have an irreducible family $\Gamma(a,b)$ of smooth curves in $\PP^5$ lying on smooth  quartic surface scrolls for each $(a,b)\in \{(3,3), (5,-5)\}$ with 
 \begin{align*}
 \dim\Gamma(a,b)&=\dim|aH+bL|+\dim\h{S}(r)\\&=\frac{a \left(a +1\right) (r-1)}{2}+\left(a +1\right) \left(b +1\right)-1+(r+3)(r-1)-3
 \end{align*}
  by \eqref{sdn},  \eqref{sf} and hence
 \begin{equation}\label{dimgg}\dim\Gamma(3,3)=68>\h{X}(15,16,5), 
 \dim\Gamma(5,-5)=64>\h{X}(15,16,5).
 \end{equation}
 For simpler notation, we set 
 \begin{equation*}\label{dimg}
 \Gamma _1:={\Gamma}(3,3), \Gamma_2:= \Gamma(5,-5).
 \end{equation*}
  \begin{remark}\label{ab1}
 $X\in \Gamma _1$ has a unique  $g^1_3$ by the Castelnuovo-Severi inequality. By the same reason,  $X$ has no complete base-point-free $g^1_x$ for $x=4,5, 6,7$ and is not  a double covering or a triple covering of an elliptic curve.
 \end{remark}
 
\begin{remark}\label{baba1}
We recall that the smooth rational normal surface scrolls in $\PP^5$ are either an image of an embedding of $\FF_0$ or an image of an embedding of $\FF_2$. The image of $\FF_2$ are limits of the image of $\FF_0$ and this phenomena is carried over to the curves lying on them. $\FF_0$ is isomorphic to a smooth quadric surface $Q\subset\PP^3$ and with this isomorphism $\Oo_{\FF_0}(1) \cong \Oo_Q(1,2)$.
\vni
(i)
Take $X\in \Gamma_1$.  If $S\cong \FF_0$, $\Oo_S(H)\cong\Oo_Q(1,2)$, $\Oo_S(L)\cong\Oo_Q(0,1)$ hence $X\in |3H+3L|=|\Oo_Q(1,2)^{\otimes 3}\otimes\Oo_Q(0,1)^{\otimes 3}|=|\Oo_Q(3,9)|$. 
\vni
(ii) If $X\in \Gamma_1$ and $S\cong \FF_2$,
$X\in |3H+3L|=|\Oo_{\FF_2}(3h+12f)|$.

\vni
(iii) For  $X\in \Gamma_2$ with $S\cong \FF_0$, $X\in |5H-5L|=|\Oo_Q(5,5)|$. Thus $X$ has exactly two $g^1_5$'s; cf. \cite[Corollary 1]{Martens}.
\vni
(iv)
For $X\in \Gamma _2$ with $S\cong \FF_2$, $X\in |5H-5L|=|\Oo_{\FF_2}(5h+10f)|$.  Thus $X$ has only one $g^1_5$ by \cite[Corollary 1]{Martens}). 
\end{remark}

 \begin{lemma}\label{bab1}
 We have $h^1(\Oo_X(2)) =0$ for all $X\in \Gamma_1$.
 \end{lemma}
 
 \begin{proof}
 Note that  $\deg (\Oo_X(2)) =\deg (K_X)$ for $X\in \Gamma_1$. Suppose $S\cong\FF_0$. By Remark \ref{baba1} (i), $X\in |\Oo_Q(3,9)|$. From the standard exact sequence
$$0\rightarrow\Oo_Q(-1,-5)\rightarrow \Oo_Q(2,4)\rightarrow \Oo_X(2,4)\rightarrow 0$$
and  by $h^0(\Oo_Q(-1,-5))=h^1(\Oo_Q(-1,-5))=0$,  we have $h^0(X, \Oo_X(2))= h^0(X,\Oo_X(2,4))
 =h^0( \Oo_Q(2,4))=15\neq g$.  The case $S\cong\FF_2$ is similar.
 \end{proof}
 
  \begin{lemma}\label{bab2} Let $X\in \Gamma_2$.
  \begin{itemize}
\item[(i)] If $X\subset S\cong\FF_0$, $h^1(\Oo_X(2)) =0$. 
\item[(ii)] If $X\subset S\cong\FF_2$, $h^1(\Oo_X(2)) =1$ and $K_X\cong\Oo_X(2)$.
\end{itemize}
 \end{lemma}
 \begin{proof}
 (i) If $S\cong\FF_0$, $h^1(\Oo_X(2))=0$ follows from the same computation as in Lemma \ref{bab1}.
\vni
(ii) For the case $S\cong\FF_2$, we have $X\in |\Oo_{\FF_2}(5h+10f)|$; Remark \ref{baba1} (iii). The exact sequence
$0\rightarrow\Oo_{\FF_2}(-3h-4f)\rightarrow \Oo_{\FF_2}(2h+6f)\rightarrow \Oo_X(2h+6f)\rightarrow 0$ and the cohomology of line bundles on $\FF_2$ (\cite[Proposition 2.3]{laf}) yield
 \begin{align*}h^0(X, \Oo_X(2))&=h^0(\FF_2, \Oo_{\FF_2}(2h+6f))+h^1(\FF_2,\Oo_{\FF_2}(-3h-4f))\\&=h^0(\FF_2, \Oo_{\FF_2}(2h+6f))+h^1(\FF_2,\Oo_{\FF_2}(h))=15+1=16.
 \end{align*} Therefore $h^1(\Oo_X(2)) =1$ by Riemann-Roch.
 \end{proof}
 \begin{lemma}\label{obo1}
For  $X\in \Gamma_1\cup\Gamma_2$,  $h^1(\PP^5,\Ii_X(3)) =2$ and $h^1(\PP^5,\Ii_X(t))=0$ for all $t\ge 4$.
 \end{lemma}
\begin{proof} Since $S$ is ACM, $h^1(\PP^5,\Ii_X(t)) =h^1(S,\Ii_{X,S}(t))$ for all $t\in \NN$. 
\vni
(a) Take $X\in \Gamma_1$:
\begin{enumerate} 
\item[(a-1)] Assume $S\cong Q$. 
Since $X\in |\Oo_Q(3,9)|$, $h^1(Q,\Ii_{X,Q}(t)) = h^1(Q,\Oo_Q(t-3,2t-9))$. We have  $h^1(Q,\Oo_Q(0,-3)) =2$
and $h^1(Q,\Oo_Q(t-3,2t-9)) =0$ for all $t\ge 4$ by the K\"{u}nneth formula.
\item[(a-2)]
Assume $S\cong \FF_2$.  We have $\Oo_S(1) \cong \Oo_{\FF_2}(h+3f)$, $X\in |\Oo_{\FF_2}(3h+ 12f)|$ and $\Ii_{X,S}(t) \cong \Oo_{\FF_2}((t-3)h+(3t-12)f)$. 
For $t=3$,  $h^1(\PP^5,\Ii_X(3)) =h^1(\FF_2,\Oo_{\FF_2}(-3f)) = 2$.
For $t\ge 4$, $h^1(\PP^5,\Ii_X(t))=h^1(\FF_2, \Oo_{\FF_2}((t-3)h+(3t-12)f)) = 0$; cf.
\cite[Proposition 2.3]{laf}. .
\end{enumerate}
\vni (b) Take $X\in \Gamma_2$:
\begin{itemize}
\item[(b-1)] If $S\cong Q$, we have $X\in |\Oo_Q(5,5)|$ and $\Ii_{X,S}(t) \cong \Oo_Q(t-5,2t-5)$.
\item[(b-2)] If $S\cong \FF_2$, we have $X\in |\Oo_{\FF_2}(5h+ 10f)|$ and $\Ii_{X,S}(t)\cong \Oo_{\FF_2}((t-5)h+(3t-10)f)$.
\end{itemize}
The verification for this case (b) is similar  and we omit the routine.
\end{proof}

\vni
 (B) $\deg S=5$ case: 
 We set 
$$\Gamma_3:=\{X\in\Hh_{15,15,5}| X\subset S, \deg S=5\}.$$We recall the following well known fact regarding surfaces
of degree $5$ in $\PP^5$.
\begin{remark}\label{class5} 
Let $S\subset\PP^5$ be a quintic surface. By the classification of quintic surfaces in $\PP^5$, 
$S$ is one of the following;  
\begin{itemize}

\item[(i)] a del Pezzo surface possibly with finitely many isolated double points

\item[(ii)] a cone over a smooth quintic elliptic curve in $\PP^4$

\item[(iii)] a cone over a rational quintic curve (either smooth or singular) in $\PP^4$ 

\item[(iv)] an image of a projection into $\PP^5$ of a surface  $\w{S}\subset\PP^6$ of minimal degree $5$ with center of projection $p\notin\w{S}$.

\end{itemize}
We now assume that there is a smooth curve $X\subset S$ with $\deg X=15$ and genus $g=16$.
We remark that the  last case (iv) is not possible under this assumption; $X$ is linearly normal since  $\pi(15,6)=13$. For the case (iii), we have either $\dim\mathrm{Sing}(S)=0$ or $S$ has a double line. In both cases, $S$ is 
the image of a linear projection of a cone $\w{S}\subset\PP^6$  over a rational normal curve $\w{C}\subset\w{H}\cong\PP^5$ with center of projection $p\in\w{H}\setminus\w{C}$. 
This is not 
possible under the existence of a linearly normal  $X\subset S$. 
\end{remark}
\begin{remark}\label{a1.1} (i)
By the preceding discussion, the assumption of the following Lemma \ref{a1} - $S$ is a quintic surface containing $X\in\HO{15,16,5}$ - implies that either $S$ is a possibly singular del Pezzo surface or it is a cone over a linearly normal elliptic curve of $\PP^4$. Singular del Pezzo surface of degree $5$ are described in \cite[\S 8.5.1]{Dolgachev}. 
They form an irreducible family. 

(ii) From Proposition \ref{b1}, we recall that  the set of all $X\in \HO{15,16,5}$ contained in a singular del Pezzo surface are limits of curves lying on a smooth del Pezzo surface. Therefore in order to identify possible irreducible components of $\HO{15,16,5}$ whose general element lies on a smooth quintic surface,  it is sufficient to study the general ones, i.e. the ones contained in a blowing -up of $\PP^2$ at $4$ distinct points in general position. 
\end{remark}
We use the following simple observation several times. 
\begin{remark}\label{ab00}
Fix any surface $S\supset X$ such that $\deg (S) \le 5$. Let $M\subset \PP^5$ be a quadric hypersurface containing $X$. If $S\nsubseteq M$ then $\deg (M\cap S)\le 10<\deg X$ and it follows that
$|\Ii_X(2)| =|\Ii_S(2)|$
if $\deg X>10$. 
\end{remark}

\begin{lemma}\label{a1}
We choose a smooth $X\in\HO{15,16,5}$ and assume that $X$ lies on an irreducible quintic surface $S\subset\PP^5$.
Then 
\begin{itemize}
\item[(i)] $X$ is the complete intersection of $S$ and a cubic hypersurface and 

\item[(ii)] $K_X \cong \Oo_X(2)$, $X$ is ACM,  $h^0(\PP^5,\Ii_X(2))=5$ and $h^0(\PP^5,\Ii_S(t)) =\binom{t+5}{5} -15t+15$ for all $t\ge 3$.
\end{itemize}
\end{lemma}
\begin{proof}
By Remark \ref{a1.1} we may assume that $S$  is either smooth or with finitely many singular points such that the general hyperplane section of $S$ is a smooth linearly normal elliptic curve in $\PP^4$. 
Fix a general hyperplane $H\subset \PP^5$. By the assumption, $E:= S\cap H$ is a smooth linearly normal quintic elliptic curve in $H$.  Since $p_a(E)=1=\pi(5,4)$,  $E$ is ACM, i.e. $h^1(H,\Ii_{E,H}(t)) =0$ for all $t\in \ZZ$; cf. \cite[Theorem 3.7, p. 87]{he}. For an integer $t$, we consider the exact sequence 
\begin{equation}\label{eqtnew2}
0 \to \Ii_S(t-1) \to \Ii_S(t)\to \Ii_{E,H}(t) \to 0.
\end{equation} 
Note that $S$ is ACM (\cite[Theorem 1.3.3]{juan}), i.e. $h^1(\PP^5,\Ii_S(x))=0$ for all $x\in \ZZ$. 
Since  $h^1(H,\Ii_{E,H}(t)) =0$, we have $h^0(H,\Ii_{E}(t)) = \binom{4+t}{4}-5t$ for all $t\ge 0$ by Riemann-Roch. Since $h^1(\PP^5, \Ii_S(t-1)) =0$ for all $t\ge 0$ and $h^0(\PP^5, \Ii_S(1))=0$, the long cohomology exact sequence of \eqref{eqtnew2}  gives
$h^0(\PP^5,\Ii_S(2)) =5$ and then 
\begin{equation}\label{cubic} h^0(\PP^5,\Ii_S(3)) = h^0(\PP^5,\Ii_S(2))+h^0(H,\Ii_{E}(3))=5+20=25.
\end{equation}
Since $\deg (\Oo_X(3)) =45>30 =\deg K_X$,  
$h^0(X, \Oo_X(3))=30$ and therefore $$h^0(\PP^5,\Ii_X(3))\ge h^0(\PP^5,\Oo_{\PP^5}(3))-h^0(X,\Oo_X(3))=\tbinom{8}{3} -h^0(X,\Oo_X(3))=26.$$
By \eqref{cubic},  $h^0(\PP^5,\Ii_S(3)) < h^0(\PP^5,\Ii_X(3))$ and there is a cubic hypersurface $W\subset\PP^5$ containing $X$ such that  $W\nsupseteq S$, the curve $S\cap W$ has $\deg S\cap W=\deg S\cdot\deg W=15$.

We want to prove that $X$ is the scheme-theoretic intersection of $S$ and $W$, i.e. $S\cap W$ is not the union of $X$ and some zero-dimensional scheme. First of all, the scheme-theoretic intersection $S\cap W$ is ACM since $S$ is ACM, $W$ is given by a single equation
and $\dim S\cap W< \dim S$ (\cite[Th. 2.1.3]{bruh},\cite[Prop. 18.13]{eisen1}, \cite[Ex. 17.4 with $\nu=1$ and $n=1$]{matsu}). Since $S\cap W$ is ACM, it has no embedded component
(\cite[Th. 2.1.2(a)]{bruh}\cite[Cor. 18.10]{eisen1},\cite[Th. 141]{kap}), i.e. $S\cap W =X$ scheme-theoretically. 

By Remark \ref{ab00},  $h^0(\PP^5,\Ii_X(2)) =h^0(\PP^5,\Ii_S(2)) =5$.  Since $X$ is ACM,
\begin{align*}
h^0(\Oo_X(2))&= h^0(\PP^5,\Oo(2))-h^0(\PP^5,\Ii_X(2))+h^1(\PP^5,\Ii_X(2))\\
&=h^0(\PP^5,\Oo(2))-h^0(\PP^5,\Ii_S(2))=16.
\end{align*}
Since $\deg (\Oo_X(2)) =\deg (K_X)$, Riemann-Roch gives $\Oo_X(2)\cong K_X$. 
For $t\ge 3$  we have $h^1(\Oo_X(t)) =0$ and by $h^1(\PP^5,\Ii_X(t))=0$
\begin{align*}
h^0(\PP^5,\Ii_X(t)) &=h^0(\PP^5,\Oo_{\PP^5}(t))-h^0(X,\Oo_X(t))+h^1(\PP^5,\Ii_X(t))\\&=\tbinom{t+5}{5} -15t+15.
\end{align*}
\end{proof}
\vspace{-2pt}
We consider the family $\Delta$ consisting of curves $X\in \HO{15,16,5}$ contained in an elliptic cone. The following lemma asserts that $\Delta$  is an irreducible family of dimension 60.
\begin{lemma}\label{aob1} 
\begin{itemize}
\item[(i)] $\Delta$ is an irreducible family of dimension $60$. 

\item[(ii)] For each $X\in \Delta$, there is a unique degree $3$ morphism $u: X\to E$ with $E$ an elliptic curve, $X$ is $6$-gonal and the $g^1_6$'s on $X$ are the pull-backs of the $g^1_2$'s on $E$ which are parametrized by the points of $E$.
\end{itemize}
\end{lemma}
\begin{proof} (i)
Let $\Delta_1$ denote the family of all degree $5$ elliptic cones in $\PP^5$. The family $\HO{5,1,4}$ of linearly normal elliptic curves in $\PP^4$ forms an irreducible family of the expected dimension $25$; cf. \cite{E2}. Hence $\dim\Delta_1=\dim\PP^5+\dim\HO{5,1,4}=30$ and $\Delta_1$  is irreducible; note that  there is a natural dominant rational map $\Delta_1\rightarrow {\PP^{5}}^*$ whose fibre over $H\in {\PP^{5}}^*$ is the irreducible Hilbert scheme $\HO{5,1,4}$ of the same dimension.
We note that the proof of Lemma \ref{a1} (i) only requires  the assumption $X\in\HO{15,16,5}$ lying on surface of degree $5$ with isolated singularities. Therefore each element  $X\in \Delta$ is a smooth complete intersections of an element of $\Delta_1$ and a cubic hypersurface. Consider the locus $$\Psi := \{(S,T)|S\nsubseteq T\} \subset \Delta_1\times\PP(H^0(\PP^5,\Oo_{\PP^5}(3))).$$
The projection $\pi_1:\Psi \rightarrow \Delta_1$ is surjective with the fiber  $$\pi^{-1}(S)\cong
\PP(H^0(\PP^5,\Oo_{\PP^5}(3))/H^0(\PP^5,\Ii_S(3))).$$
By the same computation as  \eqref{cubic} we get
\begin{equation}\label{cubic1} h^0(\PP^5,\Ii_S(3)) = h^0(\PP^5,\Ii_S(2))+h^0(H,\Ii_{S\cap H}(3))=5+20=25.
\end{equation}
By \eqref{cubic1},  we conclude that $\Psi$ is irreducible and
$$\dim\Psi=\dim\Delta_1+\dim\PP(H^0(\PP^5,\Oo_{\PP^5}(3))/H^0(\PP^5,\Ii_S(3)))=60.$$
We note that  every smooth curve $X\in\HO{15,16,5}$ is contained in a unique elliptic cone $S\in\Delta_1$ if any. This follows from the following argument. By Remark \ref{ab00}, $|\Ii_X(2))| =|\Ii_S(2))|$ if $X\subset S$. In particular  
$S$ is the base locus of $|\Ii_X(2)|$ and therefore $X$ is contained in a unique $S\in\Delta_1$.
From this, we may deduce that the natural map $\Psi \stackrel{\psi}{\rightarrow} \Delta$ where $\psi(S,T)=S\cap T$ is injective and surjective,
hence $\dim\Delta=\dim\Psi=60$.
\vni
(ii) By \eqref{conegenus3} (with $d=15, r=5$), for each $X\in \Delta$, there is a degree $3$ morphism $u: X\to E$ onto an elliptic curve induced by $\EE_5\rightarrow E$; the uniqueness of the triple covering follows from the Castelnuovo-Severi 
inequality.
$X$ is $6$-gonal and the $g^1_6$'s on $X$ are the pull-backs of $g^1_2\in W^1_2(E)\cong W_1(E)\cong\mathrm{Jac}(E)\cong E$ by the Castelnuovo-Severi inequality.
\end{proof}

Since $\dim\Delta =60=\h{X}(15,16,5)$ by Lemma \ref{aob1}, we cannot exclude the possibility that $\Delta$ may constitute a full component 
of $\HO{15,16,5}$. The following two lemmas show that  $\Delta$ is in the boundary of the component
$\Gamma_3$. 
\begin{lemma}\label{se1}
Every degree $5$ elliptic cone $T$ is a flat limit of a family of smooth del Pezzo. More precisely there is a flat family $\{S_t\}_{t\in \KK}$ of degree $5$ surfaces of $\PP^5$ such that $S_0=T$ and $S_t$ is a smooth del Pezzo for all $t\ne 0$.
\end{lemma}

\begin{proof}
From Lemma \ref{a1} (ii), $h^0(\PP^5,\Ii_T(2)) =5$ for all degree $5$ surface $T$ containing an element of $\HO{15,16,5}$.
 Given  $X\in\HO{15,16,5}$ we fix an elliptic cone $T$ with the vertex $p$.  We fix a hyperplane $H\subset \PP^5$ such that $p\notin H$ and set $E:= S\cap H$. The linearly normal elliptic curve $E$ and $p$ uniquely determine $T$.
We take a set of $4$ points $\{p_1,p_2,p_3,p_4\}\subset\PP^2$ such that no $3$ of them are collinear. We take a smooth plane cubic $E'\subset \PP^2$ containing $\{p_1,p_2,p_3,p_4\}$, $E'\stackrel{u}{\cong} E$ such that $\Oo_{E'}(3)(-p_1-p_2-p_3-p_4)\cong u^\ast(\Oo_E(1))$. 
This can be done by taking an isomorphism $u: E'\to E$ first, where $u^{-1}$ is an arbitrary embedding of $E$ as a plane cubic, taking $3$ general $o_1,o_2,o_3\in E'$, then taking as $o_4$ the unique  point of
$E'$ such that  $\Oo_{E'}(3)(-o_1-o_2-o_3-o_4)\cong u^\ast(\Oo_E(1))$. Blowing up $\PP^2$ at $\{o_1,o_2,o_3,o_4\}$, we get a smooth del Pezzo $S\subset \PP^5$ such that $S\cap H =E =T\cap H$.
We take homogeneous coordinates $x_0,x_1,\cdots,x_5$ such that $H =\{x_0=0\}$ and $p =(1:0:0:0:0:0)$.
We take a bases $$\{ q_1(x_0,\cdots,x_5), q_2(x_0,\cdots,x_5),\cdots, q_5(x_0,\cdots ,x_5)\}\subset H^0(\PP^5,\Ii_S(2)).$$ 
For any $t\in \KK$  (the base field) and $i=1,2,3,4,5$ we set $$q_{i,t}(x_0,x_1,x_2,x_3,x_4,x_5):= q_i(tx_0,x_1,x_2,x_3,x_4,x_5).$$ For $t=1$, these $5$ quadratic forms generate $H^0(\PP^5,\Ii_S(2))$, while for $t=0$, they span $H^0(\PP^5,\Ii_T(2))$.
For any $t\in \KK\setminus \{0\}$ consider the automorphism of $\PP^5$ sending $x_0\mapsto tx_0$ and $x_i\mapsto x_i$ for $i=1,\cdots,5$. For each $t\ne 0$, the common zero locus the forms  $q_{i,t}(x_0,\cdots,x_5)$, $i=1,\cdots,5$, is a surface $S_t$ projectively equivalent to $S_1=S$ while $S_0=T$. 
The family $\{S_t\}_{t\in \KK}$ is flat, because all $S_t$ have the same Hilbert polynomial (\cite[Theorem III.9.9]{Hartshorne}).
Thus $T=S_0$ is a flat limit of del Pezzo surfaces $S_t$ .
\end{proof}

\begin{lemma}\label{se2}
$\Delta$ is in the closure of $\Gamma_3\setminus \Delta$, $\Gamma_3$ is irreducible and
$\dim\Gamma_3=65$.
\end{lemma}

\begin{proof}
Fix $X\in \Delta$ and let $T$ denote the degree $5$ elliptic cone containing $X$. By Lemma \ref{se1} there is a a flat family $\{S_t\}_{t\in \KK}$ of degree $5$ surfaces of $\PP^5$ such that $S_0=T$ and $S_t$ is a smooth del Pezzo for all $t\ne 0$.
Lemma \ref{a1} shows that $X = T\cap W$ for some cubic hypersurface $W$. Since $X$ is smooth, $W$ is transversal to $T$. Since smoothness
is an open condition, there is an open neighborhood $U\subset \KK$ containing $0$ such that $W$ is transversal to $S_t$ for all $t\in U$. For each $t\in U\setminus \{0\}$, we have $S_t\cap W \in \Gamma _3\setminus \{S_0\}$. Therefore $X\subset T$ is a flat limit of the curves $S_t\cap W$, $t\in U\setminus \{0\}$.

A standard computation  in the proof of Theorem \ref{main15} (B-iii) shows that $X\in\Gamma_3\setminus\Delta$ contained in a fixed smooth del Pezzo surface $\PP^2_4\stackrel{|(3:1^4)|}{\hookrightarrow}S\subset\PP^5$ is in $|(9;3^4)|=|3(3:1^4)|$. Since no non-trivial automorphism of $\PP^2$
fixes a set of $4$ points in general position, the open subset of $\Gamma_3$ formed by curves lying on a smooth del Pezzo is irreducible and $$\dim\Gamma_3=\dim \mathrm{Aut}(\PP^5)+\dim|(9;3^4)|=35+\tbinom{9+2}{2}-1 -4\tbinom{3+1}{2}=65.$$
\end{proof}
\begin{proposition}\label{aa13}
A general $X\in \Gamma_3$ is $6$-gonal.
\end{proposition}

\begin{proof}
Take $Y\in \Gamma_3$ contained in a smooth del Pezzo surface, i.e. assume that $Y$ is a normalization of a degree $9$ plane curve having ordinary triple points at the four general points $p_1,p_2,p_3,p_4$ as its only singularities.
The pencil of lines through each $p_i$ induces a base-point-free $g^1_6$.  A fifth base-point-free $g^1_6$ on $Y$ is induced by the pencil of conics containing $\{p_1,p_2,p_3,p_4\}$.
Hence Lemma \ref{se2} gives that a general element of $\Gamma _3$ has a base-point-free $g^1_6$. By Lemma \ref{aob1} (ii), $X\in\Delta\subset\Gamma_3$ is $6$-gonal. 
By lower semi continuity of gonality, a general element of $\Gamma_3$ has gonality at least $6$ and we are done.
\end{proof}
\begin{proof}[Proof of Theorem \ref{main16}:]
We saw that $\HO{15,16,5}$ is the union of $3$ pairwise disjoint irreducible families $\Gamma_1$, $\Gamma _2$ and $\Gamma_3$. Dimensions of $\Gamma_i$ are computed and the gonality of elements in each $\Gamma_i$ has been determined; cf. dimension count \eqref{dimgg}, Lemma \ref{se2}, 
Remarks \ref{ab1} \& \ref{baba1} and Proposition \ref{aa13}. 

Recall that each element of $\Gamma _3$ is ACM by Lemma \ref{a1}, while no element of $\Gamma _1\cup \Gamma_2$ is ACM by Lemma \ref{obo1}. Hence by upper semicontinuity for cohomology, no element
of $\Gamma_3$ is a limit of elements of $\Gamma_1\cup \Gamma _2$.  

On the other hand, by Lemma \ref{bab1} and Lemma \ref{bab2}, $h^1(X,\Oo_X(2)) =0$ for all $X\in \Gamma _1$, while $h^1(X,\Oo_X(2))=1$ for all $X\in \Gamma_3$ by Lemma \ref{a1}. Again,
semicontinuity for cohomology tells that no element of $\Gamma_1$ is a limit of a family of elements of $\Gamma _3$.  

Likewise, we can deduce that a general element of $\Gamma_2$ is not a limit of a family of elements of $\Gamma _3$. 
Note that we have $\dim \Gamma_1>\dim \Gamma_2$ and each element of $\Gamma_1$ is trigonal and each element of $\Gamma_2$ is $5$-gonal. 

By (lower) semicontinuity of gonality, no element of $\Gamma_2$ is a limit of elements of $\Gamma_1$. Therefore $\HO{15,16,5}$ has exacly three distinct  irreducible components, $\Gamma _1$, $\Gamma_2$ and $\Gamma_3$.
\end{proof} 

 \subsection{The moduli map $\mu: \HO{15,16,5}\to \Mm_{16}$} 
 Let $\mu: \HO{15,16,5}\to \Mm_{16}$ denote the  moduli map. Since elements of $\Gamma_1$, $\Gamma_2$ and $\Gamma _3$ have different gonalities, for each $X\in \Gamma _i$ we have
 $\mu^{-1}(\mu(X)) \subset \Gamma_i$.
 For $i=1,2$ let $\Gamma _{i,0}$ denote the non-empty open subset of $\Gamma _i$ formed by all $X\in \Gamma_i$ contained in a minimal degree surface isomorphic to $\FF_0$. For $i=1,2$ set $\Gamma_{i,2}:= \Gamma_i\setminus \Gamma_{i,0}$, i.e $\Gamma_{i,2}$ is the set of all $X\in \Gamma_i$ contained in a minimal degree surface isomorphic to $\FF_2$.
 
 \begin{remark}\label{mz1}
 For each $X\in \Gamma _{2,0}$ (resp. $X\in \Gamma _{2,2}$) we have $\mu^{-1}(\mu(X)) \subset \Gamma_{2,0}$ (resp. $\mu^{-1}(\mu(X)) \subset \Gamma_{2,2}$);  elements of $\Gamma_{2,0}$ have exactly two $g^1_5$ and elements of $\Gamma_{i,2}$ have only one $g^1_5$; cf. Remark \ref{baba1}.
 \end{remark}
 
 Let $C$ be a trigonal curve of genus $g\ge 5$. By  the Castelnuovo-Severi inequality, $C$ has a unique $g^1_3$. Let $R$ be the trigonal line bundle on $C$. Let $m(C)$ be the Maroni invariant of $C$, i.e. let $m(C)+2$ be the first integer $t$ such that $h^0(R^{\otimes t}) \ge t+2$ \cite[eq. 1.2]{ms}. We always have $(g-4)/3 \le m(C)\le (g-2)/2$ and the canonical model of $C$ sits in a surface of degree $g-2$ in $\PP^{g-1}$ isomorphic to the Hirzebruch surface $\FF_e$ where $e:=g-2-2m(C)\ge 0$; cf. \cite[p. 172]{ms}.
 
 \begin{lemma}\label{mz2}
 If $X\in \Gamma_{1,e}$, $e\in \{0,2\}$, then $e =14-2m(X)$. 
 \end{lemma}
 \begin{proof}
 We use notation in Remark \ref{baba1}. For  $e=0$ we have $X\in |\Oo_Q(3,9)|$ and $\omega _X \cong \Oo_X(1,7)$.  
 For $e=2$, we have $X\in |\Oo_{\FF_2}(3h+12f)|$ and $\omega _X \cong \Oo_X(h+8f)$. Hence the canonical model of $X$ sits inside the image of $\FF_0$ (resp. $\FF_2$) under the morphism induced by the linear system 
 $|\Oo_Q(1,7)|$ (resp. $|\Oo_{\FF_2}(h+8f)|$).
 \end{proof}

Let $\sigma: Q\to Q$ be the automorphism which shifts the two factors of $Q=\PP^1\times\PP^1$.

\begin{proposition}\label{mz5}
For $X\in \Gamma _{2,0}$, we have $$\mu^{-1}(\mu(X)) =\mathrm{Aut}(\PP^5)X\cup \mathrm{Aut}(\PP^5)\sigma(X) \mathrm{~~and ~~}\dim \mu(\Gamma_{2,0}) = 29.$$
\end{proposition}

\begin{proof}
 For the first assertion,  we observe that $\mu^{-1}(\mu(X))\subset \Gamma_{2,0}$; Remark \ref{mz1}. We use Corollary \ref{z5.010} and Remark \ref{z5.011} to conclude that $$\mu^{-1}(\mu(X)) =\mathrm{Aut}(\PP^5)X\cup \mathrm{Aut}(\PP^5)\sigma(X).$$
Therefore
$\dim\mu(\Gamma_{2,0})=\dim\Gamma_2-\dim\Aut (\PP^5)=29.$
\end{proof}

\begin{proposition}\label{mz6}
Fix $X\in \Gamma _{1,0}$ and $X_1\in \Gamma_{1,2}$. Then we have $$\mu^{-1}(\mu(X)) = \mathrm{Aut}(\PP^5)X,  ~\mu^{-1}(\mu(X_1)) = \mathrm{Aut}(\PP^5)X_1$$ $$\dim \mu(\Gamma_{1,0}) =\dim \mu(\Gamma_1)) = 33,  ~\dim \mu(\Gamma_{1,2}) =31.$$
\end{proposition}
\begin{proof}
Lemma \ref{mz2} implies $\mu(\Gamma_{1,0})\cap \mu(\Gamma_{1,2}) =\emptyset$. 
Any isomorphism between two non-hyperelliptic curves $C_1, C_2$ of genus $g\ge 5$ induces a projective automorphism $\phi\in\Aut(\PP^{g-1})$ such that $\phi(C_1^\kappa)=C_2^\kappa$ where $C_i^\kappa\subset\PP^{g-1}$ is the canonical model of $C_i$. Now assume that $C_1$ (and hence $C_2$) is trigonal and call $T_i\subset \PP^{g-1}$ ($i=1,2$), the base locus of $|\Ii_{C_i^\kappa}(2)|$. Obviously $\phi(T_1)=T_2$. Up to  $\Aut(\PP^{g-1})$,  we may assume $T_1=T_2$.
Since $T_i$ is a  surface of minimal degree $g-2$ containing the trigonal curve $C_i^\kappa$, we have $T_i\cong \FF_e$ with $e =g-2-2m(C_i)$. 
Therefore we may deduce that isomorphism between two curves induces an automorphism of $\FF_e$.
By Remark \ref{z5.011},  minimal degree surface scrolls in $\PP^5$ which are isomorphic as abstract variety are also projectively equivalent,
hence we get the result.
\end{proof}
\vspace{-18pt}
\section{Curves of genus $g=17$ and $g=18$}
In this section we treat the two remaining cases $g=17$ and $g=18$. We first prove that there is no smooth curve of degree $d=15$ and genus $g=17$ in $\PP^5$. 

\begin{proposition}\label{main17}$\HO{15,17,5}=\emptyset$. 
\end{proposition}
\begin{proof}Since $\pi_1(15,5)=16<g=17$, $X\subset S\subset\PP^5$ with $\deg S=4$ by \cite[Th. 3.15]{he} and  $S$ is a smooth rational normal scroll  by Remark \ref{cone1} (c).  Assume $X\in|aH+bL|$. However there is no pair of integers $(a,b)$ satisfying the degree and genus formula \eqref{sd} for $(d,p_a(X))=(15,17)$.
\end{proof}

\subsection{Irreducibility of $\HO{15,18,5}$ and the moduli map $\mu: \h{H}_{15,18,5}\to \Mm_{18}$ }
\begin{proposition}\label{maing=18} $\HO{15,18,5}$ is irreducible, $\dim \HO{15,18,5}=68$ and every smooth $X\in\HO{15,18,5}$ is $4$-gonal with a unique $g^1_4$.
\end{proposition}
\begin{proof}
The irreducibility follows directly from \cite[Corollary 3.16, p. 100]{he}.
Since $X$ is an extremal curve, $X\subset S\subset\PP^5$ where $S$ is a smooth rational normal scroll by Remark \ref{cone1} (c). 
By solving 
\eqref{sd}  for $p_a(X)=18$, 
we have $X\in |4H-L|$. By \eqref{sdn} and \eqref{sf}, we have $$\dim\HO{15,18,5}=\dim|4H-L|+\dim\h{S}(5)=68.$$
$X$ has a $g^1_4$ cut out by the ruling of the scroll which is unique by the Castelnuvo-Severi inequality.
\end{proof}
Since $\deg (X)>2\deg(S)$, the theorem of Bezout gives $|\Ii_X(2)|=|\Ii_S(2)|$.
Since $S$ is cut out by quadrics,
$S$ is the base locus of $|\Ii_X(2)|$ and therefore $X\in\HO{15,18,5}$ is contained in a unique minimal degree surface; cf. \cite[p. 120]{ACGH}. There are two non isomorphic rational normal scrolls in $\PP^5$, which are the images of $\FF_0$ and $\FF_2$ under the morphism induced by appropriate very ample linear systems. Let $\Lambda_0$ ($\Lambda_2$ resp.) denotes the locus of smooth curves in $\HO{15,18,5}$ 
contained in $\FF_0$ ($\FF_2$ resp.). 
Since each $X\in \HO{15,18,5}$ is contained in a unique minimal degree surface, $\Lambda _2\cap \Lambda _0=\emptyset$. However we want to prove something stronger: we want to prove that no element of $\Lambda_0$ is isomorphic to an element of $\Lambda_2$.
We will use the following lemma which asserts that the  first scrollar invariant of the unique $g^1_4$ on $X\in \HO{15,18,5}$ detects the integer $e\in \{0,2\}$ such that $X\in \Lambda_e$.

\begin{lemma}\label{le1}
Take $e\in \{0,2\}$ and $X\in \Lambda _e$. Let $R$ be the only degree $4$ line bundle on $X$ such that $h^0(X,R)=2$. Let $c$ be the minimal integer such that $h^0(R^{\otimes c}) \ge c+2$, the first scrollar invariant of $|R|$. Then we have the following two cases for the integer $c$:
\[
\begin{cases}
 ~~e=0,  ~c=7, ~h^0(R^{\otimes7})=11\\

~e=2, ~c= 5, ~h^0(R^{\otimes 5})=7.
\end{cases}
\]
\end{lemma}
\begin{proof}
If $e=0$ we have $X\in |\Oo_Q(4,7)|$ with $R =\Oo_X(0,1)$. Fix an integer $t\ge 1$. From the cohomology sequence of the exact sequence
\begin{equation*}\label{eqle1}
0 \to \Oo_Q(-4,t-7)\to \Oo_Q(0,t)\to \Oo_X(0,t)\to 0,
\end{equation*}
we have $h^0(R^{\otimes t}) =  h^0(Q,\Oo_Q(0,t))+h^1(Q,\Oo_Q(-4,t-7))$ and hence 
$h^0(R^{\otimes7})=11$ and $c=7$.
For $e=2$ we have $X\in |4H-L|=|4h+11f|$ with $R =\Oo_X(f)$.  We use the exact sequence
\begin{equation*}\label{eqle2}
0 \to \Oo_{\FF_2}(-4h+(t-11)f)\to \Oo_{\FF_2}(tf)\to \Oo_X(tf)\to 0,
\end{equation*}
to get $h^0(R^{\otimes 5})=7$ and $c=5$. We omit routine computation.
\end{proof}

The following is an immediate consequence of Lemma \ref{le1}.
\begin{proposition} No smooth element of $\Lambda_2$ is isomorphic to a smooth element of $\Lambda_0$.
\end{proposition}

\begin{proposition}\label{z6}
Take $X\in \Lambda_0$. Then:

\begin{itemize}
\item[(a)] $X$ is $4$-gonal with a unique $g^1_4$, no base-point-free $g^1_c$ for $5\le c\le 6$ and a unique $g^1_7$. 

\item[(b)]   $\mu^{-1}(\mu(X)) =\mathrm{Aut}(\PP^5)X$. 

\item[(c)]  $\dim \mu(\HO{15,18,5}) = \dim \mu(\Lambda_0) =33$.
\end{itemize}
\end{proposition}

\begin{proof}
(a) follows from Proposition \ref{z5}. 
For (b), we fix any $\tilde{X}\in \HO{15,18,5}$ isomorphic to $X$.  Lemma \ref{le1} gives $\tilde{X}\in \Lambda_0$. Let $\tilde{S}$ (resp. $S$) be the minimal degree surface containing $\tilde{X}$ (resp. $X$).
Thus there is $u\in \Aut(\PP^5)$ such that $u(\tilde{S}) = S$; cf. Remark \ref{z5.011}. Call $\Aut^0(S)\cong \Aut(\PP^1)\times\Aut(\PP^1)\subset\Aut(S)$ the connected component of $\Aut(S)$ containing the identity. By Corollary \ref{z5.010} there is $v\in \Aut^0(S)$
such that $v(u(\tilde{X})) =X$. Remark \ref{z5.011} provides the existence of $w\in \Aut(\PP^5)$ such that $w(S) =S$ and $w_{|S} =v$. Thus $w\circ u\in\Aut(\PP^5)$ satisfies  $w\circ u (\tilde{S})=w(S)=S$ and hence $$w\circ u (\tilde{X})=w(u(\tilde{X}))=v(u(\tilde{X}))=X.$$

\noindent(c) follows from 
\begin{align*} \dim \mu(\HO{15,18,5}) &= \dim \mu(\Lambda_0)=\dim\HO{15,18,5}-\dim\mu^{-1}(\mu(X))\\&=\dim\HO{15,18,5}-\dim\Aut(\PP^5)=68-35=33.
\end{align*}
\end{proof}
\subsection{An epilogue, a remark on the larger-than-expected components}

Throughout this paper we dealt with several Hilbert schemes with larger-than-expected components. 
Regarding the existence (or non-existence) of such components in general, the following conjecture is stated in \cite[p. 142]{h1}:

\vni{\bf Conjecture:} (i) If $\h{H}$ is any component of the Hilbert scheme $\HO{d,g,r}$ such that 
the image of the rational map $\h{H}\longrightarrow \h{M}_g$ has codimention $g-4$ or less, then
$$\dim\h{H}=\h{X}{(d,g.r)}.$$
\noindent
(ii) Right after that, it is also remarked in \cite[p. 143]{h1} that {\it ``to be honest, the available evidence suggests that the existence of a number $\beta(g)$ tending linearly to $\infty$ with $g$, such that any such component $\h{H}$ whose image in $\h{M}_g$ has codimension $\beta\le\beta(g)$ has the expected dimension; we use the function $g-4$ just for simplicity"}.

Our results discussed in this paper suggest that the function $\beta(g)=g-4$ may need to be replaced with smaller $\beta(g)$.
\vspace{-2pt}
\begin{remark} 
(i) For $(d,g,r)=(15,15,5)$, $\HO{d,g,r}=\HL{d,g,r}$ is irreducible and $\dim\HO{d,g,r}=64>\h{X}(d,g,r)=62$ by Theorem \ref{main15}. By Theorem \ref{main15}, $\codim \mu(\HO{d,g,r})=3\cdot g-3-(\dim\HO{d,g,r}-\dim\Aut(\PP^r))=13>g-4$ and hence the family $\HO{d,g,r}$ is  {\it special} with big codimension in $\h{M}_g$. Also,  this does not give an evidence to disprove the conjecture.

\vni
(ii) For $(d,g,r)=(15,16,5)$ the three components $\Gamma_i$ ($i=1,2,3$) of $\HO{d,g,r}=\HL{d,g,r}$ have dimension larger than $\h{X}(d,g,r)=60$ by Theorem \ref{main16}. By Proposition \ref{mz6}, $\codim\mu(\Gamma_1)=12=g-4$, hence providing an evidence to the contrary to the conjecture if  $\beta(g)=g-4$. 
\vni
(iii) For $(d,g,r)=(15,13,5)$, $\HO{d,g,r}$ is reducible with two components, one with the expected dimension and the other one with more than expected dimension whose image under the moduli map $\mu$ has codimension $g-4=\codim\h{M}^1_{g,3}=3g-3-(2g+1)$; Proposition \ref{g=13}. Hence the literal statement of the above conjecture turns out to be untrue if one puts $\beta(g)=g-4$. However, elements of this component are not linearly normal. There are other examples of this kind suggesting that $\beta(g)=g-4$ is rather too large to ensure the validity of the above conjecture. By \cite[Proposition 3.5]{CKP}, there exists a component 
of the Hilbert scheme with $(d,g,r)=(2g-2-2k,g,r)$ -- subject to several crude and technical numerical conditions such as $2\le \frac{g-4}{k-1},\frac{2g+2-2k}{3}-1<r\le
g-2-2k,g>(k-1)^2$ -- dominating $\h{M}^1_{g,k}$, whereas $\codim\h{M}^1_{g,k}=g+2-2k\le \beta(g)=g-4$. Again, curves in this component are not linearly normal. 
\vni
(iv) If one focus on components consisting of linearly normal curves, $\beta(g)=g-5$  would be a 
choice suggested by the example in (ii). The authors do not know of any example of a component with dimension greater-than-expected
such that the image under the moduli map $\mu$ has codimension at most $g-5$.
\end{remark}

\end{document}